\documentclass[letter]{amsart}

\usepackage{amssymb}
\usepackage{amsfonts}
\usepackage{amsmath}
\usepackage{graphicx}
\usepackage{mathrsfs}
\usepackage{dsfont}
\usepackage{amscd}
\usepackage{multirow}
\usepackage{palatino}
\usepackage{verbatim}
\usepackage[all]{xy}

\newcommand{\N}{\mathds{N}}
\newcommand{\Z}{\mathds{Z}}
\newcommand{\Q}{\mathds{Q}}
\newcommand{\R}{\mathds{R}}
\newcommand{\C}{\mathds{C}}
\newcommand{\T}{\mathds{T}}

\theoremstyle{plain}
\newtheorem{theorem}{Theorem}[section]

\newtheorem{corollary}[theorem]{Corollary}

\newtheorem{lemma}[theorem]{Lemma}
\newtheorem{proposition}[theorem]{Proposition}

\theoremstyle{definition}
\newtheorem{definition}[theorem]{Definition}

\newtheorem{notation}[theorem]{Notation}

\theoremstyle{remark}

\newtheorem{case}{Case}[theorem]
\newtheorem{example}[theorem]{Example}

\newtheorem{remark}[theorem]{Remark}

\numberwithin{equation}{section}

\begin{document}
\title[Noncommutative Solenoids]{Noncommutative Solenoids}
\author{Fr\'{e}d\'{e}ric Latr\'{e}moli\`{e}re}
\email{frederic@math.du.edu}
\urladdr{http://www.math.du.edu/\symbol{126}frederic}
\address{Department of Mathematics \\ University of Denver \\ Denver CO 80208}

\author{Judith Packer}
\email{packer@euclid.colorado.edu}
\urladdr{http://spot.colorado.edu/\symbol{126}packer}
\address{Department of Mathematics \\ University of Colorado \\ Boulder CO 80309}
\date{\today}

\begin{abstract}
A noncommutative solenoid is the C*-algebra $C^\ast(\Q_N^2,\sigma)$ where $\Q_N$ is the group of the $N$-adic rationals twisted and $\sigma$ is a multiplier of $\Q_N^2$. In this paper, we use techniques from noncommutative topology to classify these C*-algebras up to *-isomorphism in terms of the multipliers of $\Q_N^2$. We also establish a necessary and sufficient condition for simplicity of noncommutative solenoids, compute their K-theory and show that the $K_0$ groups of noncommutative solenoids are given by the extensions of $\Z$ by $\Q_N$. We give a concrete description of non-simple noncommutative solenoids as bundle of matrices over solenoid groups, and we show that irrational noncommutative solenoids are real rank zero $\mathrm{AT}$ C*-algebras.
\end{abstract}
\maketitle


\newcommand{\solenoid}{\mathscr{S}}
\newcommand{\algebra}{\mathscr{A}^{\solenoid}}
\newcommand{\primeseq}{\mathscr{P}}

\section{Introduction}

Since the early 1960's, the specific form of transformation group $C^{\ast}$-algebras given by the action of $\Z$ on the circle generated through a rotation that was an irrational multiple of $2\pi$ has sparked interest in the classification problem for $C^{\ast}$-algebras in particular and the theory of $C^{\ast}$-algebras in general.  When first introduced by Effros and Hahn in \cite{EffrosHahn67}, it was thought that these $C^{\ast}$-algebras had no non-trivial projections.  This was shown not to be the case by M. Rieffel in the late 1970's \cite{Rieffel81}, when he constructed a whole family of projections in these $C^{\ast}$-algebras, and these projections played a key role one of Pimsner's and Voiculescu's methods of classifying these $C^{\ast}$-algebras up to $\ast$-isomorphism, achieved in 1980 (\cite{Pimsner80}) by means of $K$-theory.  Since then these $C^{\ast}$-algebras were placed into the wider class of twisted $\Z^n$-algebras by M. Rieffel in the mid 1980's (\cite{Rieffel88}) and from this point of view were relabeled as \emph{non-commutative tori}.  The $\Z^n$-analogs have played a key role in the non-commutative geometry of A. Connes \cite{Connes80}, and the class of $C^{\ast}$-algebras has been widened to include twisted $C^{\ast}$-algebras associated to arbitrary compactly generated locally compact Abelian groups \cite{Echterhoff95}.  However, up to this point, the study of twisted group $C^{\ast}$-algebras associated to Abelian groups that are not compactly generated has been left somewhat untouched.

There are a variety of reasons for this lack of study, perhaps the foremost being that Abelian groups that cannot be written as products of Lie groups $\R^n$ and finitely generated Abelian groups are much more complicated and best understood by algebraists; furthermore, the study of extensions of such groups can touch on logical conundrums.  One could also make the related point that such groups require more technical algebraic expertise and are of less overall interest in applications than their compactly generated counterparts. On the other hand, it can also be said that discrete Abelian groups that are not finitely generated have begun to appear more frequently in the literature, including in algebra in the study of the two-relation Baumslag-Solitar groups, where they appear as normal Abelian subgroups, in the study of wavelets, where these groups and their duals, the solenoids, have appeared increasingly often in the study of wavelets \cite{DutkayJorgensen06,DutkayJorgensen07,Dutkay06,Baggett10,Baggett11}.  We thus believe it is timely to study the twisted C*-algebras of the groups $\Q_N^2$ where $\Q_N$ is the group of $N$-adic rational numbers for arbitrary natural number $N>1$ and in homage to M. Rieffel, we call such $C^{\ast}$-algebras \emph{non-commutative solenoids}.

In this paper, we present the classification of noncommutative solenoids up to $\ast$-isomorphism using methods from noncommutative topology. They are interesting examples of noncommutative spaces, and in particular, they can be seen as noncommutative orbit spaces for some actions of the $N$-adic rationals on solenoids, some of them minimal. Thus, our classification provides a noncommutative topological approach to the classification of these actions as well. Our work is a first step in the study of the topology of these new noncommutative spaces. Our classification result is based on the computation of the $K$-theory of noncommutative solenoids. We prove that the $K_0$ groups of noncommutative solenoids are exactly the groups given by Abelian extensions of $\Z$ by $\Q_N$, which follows from a careful analysis of such extensions. We relate the class of noncommutative solenoids with the group $\mathrm{Ext}(\Q_N,\Z)$, which is isomorphic to $\Z_N / \Z$ where $\Z_N$ is the additive group of $N$-adic integers \cite{Fuchs70}, and we make explicit the connection between $N$-adic integers and our classification problem. We also partition the class of noncommutative solenoids into three distinct subclasses, based upon their defining twisting bicharacter: rational periodic noncommutative solenoids, which are the nonsimple noncommutative solenoids, and the only ones of of type I, and are fully described as bundles of matrices over a solenoid group; irrational noncommutative solenoids, which we show to be simple and real rank zero $\mathrm{AT}$-algebras in the sense of Elliott; and last rational aperiodic noncommutative solenoids, which give very intriguing examples.

\bigskip We build our work from the following family of groups:

\begin{definition}
Let $N\in\N$ with $N>1$. The group of $N$-adic rationals is the group:
\begin{equation}\label{NadicRationals}
\Q_N = \left\{ \frac{p}{N^k} \in \Q : p\in\Z, k\in\N \right\}
\end{equation}
endowed with the discrete topology.
\end{definition}
An alternative description of the group $\Q_N$ is given as the inductive limit of the sequence: 
\begin{equation}
\begin{CD}\label{Nadic-def}
\Z @>z\mapsto Nz>> \Z @>z\mapsto Nz>> \Z @>z\mapsto Nz>> \Z @>z\mapsto Nz>> \cdots
\end{CD}
\end{equation}
From this latter description, we obtain the following result. We denote by $\T$ the unit circle $\{ z \in \C : |z|=1\}$ in the field $\C$ of complex numbers.
\begin{proposition}\label{solenoids}
Let $N\in\N$ with $N>1$. The Pontryagin dual of the group $\Q_N$ is the $N$-solenoid group, given by:
\[
\solenoid_N = \left\{ (z_n)_{n\in\N} \in \T^{\N}: \forall n \in \N \;\; z_{n+1}^N = z_n \right\}\text{,}
\]
endowed with the induced topology from the injection $\solenoid_N\hookrightarrow \T^{\N}$. The dual pairing between $\Q_N$ and $\solenoid_N$ is given by:
\[
\left< \frac{p}{N^k}, (z_n)_{n\in\N} \right> = z_k^p \text{,}
\]
where $\frac{p}{N^k} \in \Q_N$ and $(z_n)_{n\in\N}\in\solenoid_N$.
\end{proposition}

\begin{proof}
The Pontryagin dual of $\Q_N$ is given by taking the projective limit of the sequence:
\begin{equation}
 \label{solenoid-def}
\begin{CD}
\cdots @>z\mapsto z^N>>\T @>z\mapsto z^N>>\T @>z\mapsto z^N>> \T @>z\mapsto z^N>> \T
\end{CD}\text{.}
\end{equation}
using the co-functoriality of Pontryagin duality and Sequence (\ref{Nadic-def}). We check that this limit is (up to a group isomorphism) the group $\solenoid_N$, and the pairing is easily computed.
\end{proof}

Using Proposition (\ref{solenoids}), we start this paper with the computation of the second cohomology group of $\Q_N^2$. We then compute the symmetrizer group for any skew-bicharacter of $\Q_N^2$, as it is the fundamental tool for establishing simplicity of twisted group C*-algebras. The second section of this paper studies the basic structure of quantum solenoids, defined as $C^\ast(\Q_N^2,\sigma)$ for $\sigma\in H^2(\Q_N^2)$. We thus establish conditions for simplicity, and isolate the three subclasses of noncommutative solenoids. We then compute the $K$-theory of noncommutative solenoids and show that they are extensions of $\Z$ by $\Q_N$. We then prove that the $K_0$ groups of noncommutative solenoids are given exactly by \emph{all} possible Abelian extensions of $\Z$ by $\Q_N$. This section presents self-contained computations of the $\Z$-valued $2$-cocycles of $\Q_N$ corresponding to $K_0$ groups of noncommutative solenoids and a careful analysis of $\mathrm{Ext}(\Q_N,\Z)$. We then compute an explicit presentation of rational noncommutative solenoids.

In our last section, we classify all noncommutative solenoids in terms of their defining $\T$-valued $2$-cocycles. Our technique, inspired by the work of \cite{Yin86} on rational rotation C*-algebras, uses noncommutative topological methods, namely our computation of the $K$-theory of noncommutative solenoids. We also connect the theory of Abelian extensions of $\Z$ by $Q_N$ with our *-isomorphism problem.

Our work is a first step in the process of analyzing noncommutative solenoids. Questions abound, including queries about Rieffel-Morita equivalence of noncommutative solenoids and the structure of their category of modules, additional structure theory for aperiodic rational noncommutative solenoids, higher dimensional noncommutative solenoids and to what extent the Connes' noncommutative geometry can be extended to these noncommutative solenoids.

\section{Multipliers of the $N$-adic rationals}

We first compute the second cohomology group of $\Q_N^2$. A \emph{noncommutative solenoid} will mean, for us, a twisted group C*-algebra of $\Q_N\times\Q_N$ for some $N\in\N, N>1$. We shall apply the work of Kleppner \cite{Kleppner65} to determine the group $H^2(\Q_{N}^2)$ for $N\in\N,N>1$. 

\begin{theorem}\label{skew-multipliers}
Let $N\in\N,N>1$. We let:
\[ \Xi_N = \left\{ \left( \nu_n \right) : 
\nu_0 \in [0,1) \,\,\wedge\,\, (\forall n\in \N \,\, \exists k \in \{0,\ldots,N-1\}\;\; N \nu_{n+1} =  \nu_n+k )
\right\}\text{.}
\]
The set $\Xi_N$ is a group for the pointwise modulo-one addition operation. As a group, $\Xi_N$ is isomorphic to $\solenoid_N$. Let $B^{(2)}(\Q_N^2)$ be the group of skew-symmetrized bicharacters defined by:
\[
B^{(2)}(\Q_N^2) = \{ (x,y)\in\Q_N^2 \mapsto \varphi(x,y)\varphi(y,x)^{-1} : \varphi \in B(\Q_N^2) \}
\]
where $B(\Q_N^2)$ is the group of bicharacters of $\Q_N^2$. Then $\varphi \in B^{(2)}$ if and only if there exists $\alpha\in\Xi_N$ such that, for all $p_1,p_2,p_3,p_4\in\Z$ and $k_1,k_2,k_3,k_4 \in \N$, we have
\[
\varphi\left(\left(\frac{p_1}{N^{k_1}},\frac{p_2}{N^{k_2}}\right),\left(\frac{p_3}{N^{k_3}},\frac{p_4}{N^{k_4}}\right)\right) = \exp(2i\pi (\alpha_{(k_1+k_4)}p_1p_4 - \alpha_{(k_2+k_3)}p_2p_3))\text{.}
\]
Moreover, $\alpha$ is uniquely determined by $\varphi$.
\end{theorem}

\begin{proof}
If $\alpha\in\Xi_N$ then $\alpha_k\in[0,1)$ for all $k\in\N$. Indeed $\alpha_0\in[0,1)$ and if $\alpha_k\in[0,1)$ then $\alpha_{k+1} = \frac{\alpha_k+j}{N}$ with $0\leq j \leq N-1$ so $0\leq \alpha_{k+1} < 1$, so our claim holds by induction. With this observation, it becomes straightforward to check that $\Xi_N$ is a group for the operation of entry-wise addition modulo one.

By definition of $\Xi_N$, the map $e:\Xi_N\mapsto \solenoid_N$ defined by $e(\alpha)_k=\exp(2i\pi \alpha_k)$ for any $\alpha\in\Xi_N$ is a bijection, which is easily checked to be a group isomorphism.
 
\bigskip Following \cite{Kleppner65}, let $B$ be the group of bicharacters of $\Q_N^2$ and denote the group $B^{(2)}(\Q_N^2)$ simply by $B^{(2)}$.

The motivation for this computation is that, as a group, $B^{(2)}(\Q_N^2)$ is isomorphic to $H^2(\Q_N^2)$ by \cite[Theorem 7.1]{Kleppner65} since $\Q_N$ is discrete and countable. However, we will find a more convenient form of $H^2(\Q_N^2)$ in our next theorem using the following computation:

\bigskip Let $\Psi\in B^{(2)}$. Fix $\varphi\in B$ such that:
\[
\Psi : x,y \in \Q_N^2\times \Q_N^2 \longmapsto \varphi(x,y)\varphi(y,x)^{-1}\text{.}
\]
Now, the dual of $\Q_N^2$ is $\solenoid_N^2$ with pairing given in Proposition (\ref{solenoids}). The map:
\[
\frac{p}{N^k}\in\Q_N\longmapsto \varphi\left((1,0),\left(\frac{p}{N^k},0\right)\right)
\]
is a character of $\Q_N$, so there exists a unique $\zeta\in\solenoid_N$ such that:
\[
\varphi\left((1,0),\left(\frac{p}{N^k},0\right)\right) = \zeta_{k}^p
\]
for all $p\in\Z, k\in\N$. Similarly, there exists $\eta,\chi,\xi\in\solenoid_N$ such that for all $p\in\Z,k\in\N$ we have:
\begin{eqnarray*}
\varphi\left((0,1),\left(\frac{p}{N^k},0\right)\right) &=& \eta_{k}^p \\
\varphi\left((0,1),\left(0,\frac{p}{N^k}\right)\right) &=& \chi_{k}^p \\
\varphi\left((1,0),\left(0,\frac{p}{N^k}\right)\right) &=& \xi_{k}^p
\end{eqnarray*}

Using the bicharacter property of $\varphi$ again, we arrive at:
\[
\varphi\left(\left(p_1,p_2\right),\left(\frac{p_3}{N^{k_3}},\frac{p_4}{N^{k_4}}\right)\right) = \zeta_{k_3}^{p_1p_3}\eta_{k_3}^{p_2p_3}\chi^{p_2p_4}_{k_4}\xi^{p_1p_4}_{k_4}\text{.}
\]
Now, since $\varphi\left(\left(\frac{1}{N^k},0\right),\left(\frac{p}{N^{k_3}},0\right)\right)^{(N^k)} = \varphi\left( (1,0),\left(\frac{p}{N^{k_3}}\right)\right)$, there exists $\nu\in\solenoid_N$ with $\nu_0 = 1$ such that:
\[
\varphi\left(\left(\frac{1}{N^k},0\right),\left(\frac{p}{N^{k_3}},0\right)\right) = \nu_k\zeta_{k+k_3}^p\text{,}
\]
where we use the property that $\zeta_{k+k_3}^{(N^k)} = \zeta_{k_3}$.
Since $\varphi(g,0)=1$ for any $g\in\Q_N^2$, we have $\nu = 1$. By the same method, we deduce:
\[
\varphi\left(\left(\frac{p_1}{N^{k_1}},\frac{p_2}{N^{k_2}}\right),\left(\frac{p_3}{N^{k_3}},\frac{p_4}{N^{k_4}}\right)\right) = \zeta_{k_1+k_3}^{p_1p_3}\eta_{k_2+k_3}^{p_2p_3}\chi^{p_2p_4}_{k_2+k_4}\xi^{p_1p_4}_{k_1+k_4}\text{.}
\]
Now, by setting all but one of $p_1,p_2,p_3,p_4$ to zero, we see that $\varphi$ determines $(\eta,\zeta,\chi,\xi)\in\solenoid_N^4$ uniquely. Thus, we have defined an injection $\iota$ from the group of bicharacters of $\Q_N^2$ into $\solenoid_N^4$ by setting, with the above notation: $\iota(\varphi)=(\zeta,\xi,\eta,\chi)$. It is straightforward that this map is a bijection.

Thus, $\vartheta:\iota^{-1}\circ e^{\otimes 4} : \Xi_N^4 \rightarrow B(\Q_N^2)$ is a bijection, so there exists a unique $(\beta,\gamma,\mu,\rho)\in\Xi_N^4$ such that for all $p_1, p_2, p_3, p_4 \in \Z$ and $k_1, k_2, k_3, k_4 \in \N$, the value $$\varphi\left(\left(\frac{p_1}{N^{k_1}},\frac{p_2}{N^{k_2}}\right),\left(\frac{p_3}{N^{k_3}},\frac{p_4}{N^{k_4}}\right)\right)$$ is given by:
\[
\exp\left( 2i\pi 
\left[
\begin{array}{cc}
p_1 & p_2
\end{array}
\right] 
\left[ 
\begin{array}{cc}
\beta_{k_1+k_3} & \gamma_{k_1+k_4}\\
\mu_{k_2+k_3} & \rho_{k_2+k_4}
\end{array}
\right] 
\left[ 
\begin{array}{c}
p_3 \\ p_4
\end{array}
\right]
\right) \text{.}
\]
Thus, $\varphi\left(\left(\frac{p_3}{N^{k_3}},\frac{p_4}{N^{k_4}}\right),\left(\frac{p_1}{N^{k_1}},\frac{q_1}{N^{k_2}}\right)\right)^{-1}$ is given by:
\[ 
\exp\left( -2i\pi 
\left[
\begin{array}{cc}
p_1 & p_2
\end{array}
\right] 
\left[ 
\begin{array}{cc}
\beta_{k_1+k_3} & \mu_{k_1+k_4}\\
\gamma_{k_2+k_3} & \rho_{k_2+k_4}
\end{array}
\right] 
\left[ 
\begin{array}{c}
p_3 \\ p_4
\end{array}
\right]
\right) 
\]
after transposing the matrix multiplication as the product is a scalar. 
So $$\Psi\left(\left(\frac{p_1}{N^{k_1}},\frac{p_2}{N^{k_2}}\right),\left(\frac{p_3}{N^{k_3}},\frac{p_4}{N^{k_4}}\right)\right)$$ is:
\begin{equation}\label{skew-bicharacter}
\exp\left( 2i\pi 
\left[
\begin{array}{cc}
p_1 & p_2
\end{array}
\right] 
\left[ 
\begin{array}{cc}
0 & (\gamma-\mu)_{(k_1+k_4)}\\
(\mu-\gamma)_{(k_2+k_3)} & 0
\end{array}
\right] 
\left[ 
\begin{array}{c}
p_3 \\ p_4
\end{array}
\right]
\right)
\end{equation}
though it is not in our chosen canonical form, i.e. $\gamma-\mu$ may not lie in $\Xi_N$ --- it takes values in $(-1,1)$ instead of $[0,1)$. Let us find the unique element of $\Xi_N^4$ which is mapped by $\vartheta$ to $\Psi$. Observe that we can add any integer to the entries of the matrix in Expression (\ref{skew-bicharacter}) without changing $\Psi$. Let $n\in\N$. Set $\epsilon_n$ to be $1$ if $\gamma_n-\nu_n<0$, or to be $0$ otherwise. Let $\omega_n^1 = \epsilon_n + \gamma_n - \mu_n$ and $\omega_n^2 = (1-\epsilon_n ) + \mu_n - \gamma_n$. We check that $\omega^1,\omega^2 \in \Xi_N$ and that $\omega^1_n + \omega^2_n = 1$ for all $n\in\N$. We can moreover write:
$$\Psi\left(\left(\frac{p_1}{N^{k_1}},\frac{p_2}{N^{k_2}}\right),\left(\frac{p_3}{N^{k_3}},\frac{p_4}{N^{k_4}}\right)\right)$$ as:
\begin{equation}\label{skew-bicharacter-2}
\exp\left( 2i\pi 
\left[
\begin{array}{cc}
p_1 & p_2
\end{array}
\right] 
\left[ 
\begin{array}{cc}
0 & (\omega^1)_{(k_1+k_4)}\\
(\omega^2)_{(k_2+k_3)} & 0
\end{array}
\right] 
\left[ 
\begin{array}{c}
p_3 \\ p_4
\end{array}
\right]
\right)
\end{equation}
i.e. $\Psi = \vartheta(0,\omega^1,\omega^2,0)$. Since $\omega^1+\omega^2$ is the constant sequence $(1)_{n\in\N}$, we have in fact constructed a bijection from $\Xi_N$ onto $B^2(\Q_N^2)$ as desired.

The form for $\Psi$ proposed in the Theorem is more convenient. We obtain it by simply subtracting $1$ from $\omega^2_n$ for all $n\in\N$, which does not change the value of Expression(\ref{skew-bicharacter-2}). We thus get:

$$\Psi\left(\left(\frac{p_1}{N^{k_1}},\frac{p_2}{N^{k_2}}\right),\left(\frac{p_3}{N^{k_3}},\frac{p_4}{N^{k_4}}\right)\right) =$$
\[
\exp\left( 2i\pi 
\left[
\begin{array}{cc}
p_1 & p_2
\end{array}
\right] 
 \left[ 
\begin{array}{cc}
0 & \alpha_{(k_1+k_4)}\\
-\alpha_{(k_2+k_3)} & 0
\end{array}
 \right]
\left[ 
\begin{array}{c}
p_3 \\ p_4
\end{array}
\right]
\right)\text{.}
\]

\end{proof}

While \cite{Kleppner65} shows that $B^{(2)}(\Q_N^2)$ is, as a group, isomorphic to $H^2(\Q_N^2)$, a point of subtlety is that several elements of $B^{(2)}(\Q_N^2)$ may be cohomologous, i.e. there are in general two non-cohomologous multipliers of $\Q_N^2$ which are mapped by this isomorphism to two distinct but cohomologous multipliers in $B^{(2)}(\Q_N^2)$. 
\begin{example}
If $N=3$, then one checks that $\alpha=\left(\frac{1}{2}\right)_{n\in\N}\in\Xi_3$. This element corresponds to the element $(-1)_{n\in\N}$ in $\solenoid_3$. Now, if $\phi$ is given by Theorem (\ref{skew-multipliers}), then $\varphi\in B^{(2)}(\Q_3^2)$ is symmetric. Hence it is cohomologous to the trivial multiplier $1\in B^{(2)}(\Q_3^2)$. However, there exists two multipliers $\sigma_1,\sigma_2$ of $\Q_3^2$ which are not cohomologous, and map, respectively, to $\varphi$ and $1$, since \cite{Kleppner65} shows that there is a bijection from $H^2(\Q_3^2)$ onto $B^{(2)}(\Q_3^2)$.
\end{example}
This is quite inconvenient, and we prefer, for this reason, the description of multipliers of $\Q_N^2$ up to equivalence given by our next Theorem (\ref{multipliers}). 

\begin{theorem}\label{multipliers}
Let $N\in\N,N>1$. There exists a group isomorphism $\rho : H^2(Q_N^2) \rightarrow \Xi_N$ such that if $\sigma\in H^2(\Q_N^2)$ and $\alpha=\rho(\sigma)$, and if $f$ is a multiplier of class $\sigma$, then $f$ is cohomologous to:
\[
\Psi_\alpha\left(\left(\frac{p_1}{N^{k_1}},\frac{p_2}{N^{k_2}}\right),\left(\frac{p_3}{N^{k_3}},\frac{p_4}{N^{k_4}}\right)\right) = \exp(2i\pi \alpha_{(k_1+k_4)}p_1p_4)\text{.}
\]
\end{theorem}

\begin{proof}
Let $\delta : B(\Q_N^2) \rightarrow B^{(2)}(\Q_N^2)$ be the epimorphism from the group of bicharacters of $\Q_N^2$ onto $B^{(2)}(\Q_N^2)$ defined by:
\[
\delta(\varphi) : (x,y) \in \Q_N^2 \mapsto \varphi(x,y)\varphi(y,x)^{-1}
\]
for all $\varphi \in B(\Q_N^2)$. We shall define a cross-section $\mu : B^{(2)}(\Q_N^2) \rightarrow B(\Q_N^2)$, i.e. a map such that $\delta\circ\mu$ is the identity on $B^{(2)}(\Q_N^2)$.

For $\varphi \in B^{(2)}(\Q_N^2)$, by Theorem (\ref{skew-multipliers}) there exists a unique $\alpha\in\Xi_N$ such that:
$$\varphi\left(\left(\frac{p_1}{N^{k_1}},\frac{p_2}{N^{k_2}}\right),\left(\frac{p_3}{N^{k_3}},\frac{p_4}{N^{k_4}}\right)\right) =$$
\[
\exp\left( 2i\pi 
\left[
\begin{array}{cc}
p_1 & p_2
\end{array}
\right] 
 \left[ 
\begin{array}{cc}
0 & \alpha_{(k_1+k_4)}\\
-\alpha_{(k_2+k_3)} & 0
\end{array}
 \right]
\left[ 
\begin{array}{c}
p_3 \\ p_4
\end{array}
\right]
\right)\text{.}
\]
Define $\mu(\varphi)=\Psi_\alpha$. We then check immediately that $\delta\circ\mu$ is the identity.

Now, denote by $\zeta:H^2(\Q_N^2)\rightarrow B^{(2)}(\Q_N^2)$ the isomorphism from \cite{Kleppner65}. If $f$ and $g$ are two multipliers of $\Q_N^2$, then $\zeta(f)=\zeta(g)\in B^2(\Q_N^2)$ if and only if $f,g$ are cohomologous. So $\mu(\zeta(f))$ is cohomologous to $f$ as desired.
\end{proof}

\begin{remark}
We thus have shown that $H^2(\Q_N^2)$ is isomorphic to $\solenoid_N$ for all $N\in\N,N>1$. However, we find the identification of $H^2(\Q_N^2)$ with $\Xi_N$ more practical in our proofs.
\end{remark}

The simplicity of twisted group C*-algebras is related to the symmetrizer subgroup of the twisting bicharacter. We thus establish, using the notations introduced in Theorem (\ref{skew-multipliers}), a necessary and condition for the triviality of the symmetrizer group of multipliers of $\Q_N$ for $N\in\N,N>1$. As our work will show, it is in fact fruitful to invest some effort in working with a generalization of the group $\Xi_N$ based upon certain sequences of prime numbers. 

\begin{definition}
The set of all sequences of prime numbers with finite range is denoted by $\primeseq$. 
\end{definition}

As a matter of notation, if $\Lambda\in\primeseq$ then its $n^{\mathrm{th}}$ entry is denoted by $\Lambda_n$, so that $\Lambda=(\Lambda_n)_{n\in\N}$.

\begin{definition}
Let $\Lambda\in\primeseq$. For all $k\in\N,K>0$ we define $\pi_k(\Lambda)$ as $\prod_{j=0}^{k-1} \Lambda_j$, and $\pi_0(\Lambda)=1$. The set $\{\pi_k(\Lambda):k\in\N\}$ is denoted by $\Pi(\Lambda)$. Note that $\pi$ defines a strictly increasing map from $\N$ into $\Pi(\Lambda)$, whose inverse will be denoted by $\delta$.
\end{definition}

Periodic sequences form a subset of $\primeseq$, and if we impose a specific ordering on the prime numbers appearing in the smallest period of such a periodic sequence, we can define a natural embedding of $\N \setminus \{0,1\}$ in $\primeseq$. We shall use:

\begin{notation}
Given two integers $n$ and $m$, the remainder for the Euclidean division of $n$ by $m$ in $\Z$ is denoted by $n \mod m$. On the other hand, given $H\subset G$ and $x,y\in G$, then $x\equiv y \mod H$ means that $x$ and $y$ are in the same $H$-coset in $G$.
\end{notation}

\begin{definition}
Let $\Lambda\in\primeseq$ be a periodic sequence. If $T$ is the minimal period of $\Lambda\in\primeseq$, we define $\nu(\Lambda)$ to be the natural number $\pi_{T-1}(\Lambda)=\prod_{n=0}^{T-1} \Lambda_n$. Conversely, if $N\in\N$ and $N>1$, we define $\Lambda(N)\in\primeseq$ as the sequence $(\lambda_{n \mod \Omega(N)})_{n\in\N}$ where $\Omega(N)$ is the number of primes in the decomposition of $N$, $\lambda_0\leq\ldots\leq\lambda_{\Omega(N)-1}$ are prime and $N=\prod_{j=0}^{\Omega(N)-1} \lambda_j$. Thus in particular, $\nu(\Lambda(N))=N$. 
\end{definition}

A central family of objects for our work is given by:

\begin{definition}
Let $\Lambda\in\primeseq$. The group $\Xi_\Lambda$ is defined as a set by:
\[
\Xi_\Lambda = \{ (\alpha_n)_{n\in\N} : \forall n\in\N\;\; \exists k\in \{0,\ldots,\Lambda_{n}-1\}\;\;\; \Lambda_n \alpha_{n+1} = \alpha_n + k \}\text{,}
\]
and with the operation of pointwise addition modulo 1.
\end{definition}

The group $\Xi_{\Lambda(N)}$ is isomorphic to $\Xi_N$, as defined in Theorem (\ref{skew-multipliers}). An explicit construction of an isomorphism is given by:
\begin{proposition}\label{XiNXiLambdaIso}
Let $N\in\N$ with $N>1$. Let $\Omega(N)$ be the minimal period of $\Lambda(N)$, i.e. the number of prime factors in the decomposition of $N$. The map:
\[
\omega : \left\{ 
\begin{array}{ccc}
\Xi_{\Lambda(N)} & \longrightarrow & \Xi_N \\
(\nu_n)_{n\in\N} & \longmapsto & (\nu_{n\Omega(N)})_{n\in\N}
\end{array}
\right.
\]
is a group isomorphism.
\end{proposition}

\begin{proof}
Let $\alpha\in\Xi_{\Lambda(N)}$. Define $\omega(\alpha)_k = \alpha_{k\Omega(N)}$ for all $k\in\N$. It is immediate to check that $\omega(\alpha)\in\Xi_N$ and, thus defined, $\omega$ is a group monomorphism. We shall now prove it is also surjective. Let us denote $\Lambda(N)$ simply by $\Lambda$.

 Let $(\nu_{n\in\N})\in\Xi_N$. Let $\eta_{n\Omega(N)} = \nu_n$ for all $n\in\N$. Let $n\in\N$. By definition of $\Xi_N$, there exists $m\in\{0,\ldots,N-1\}$ such that $N\nu_{n+1}=\nu_n + m$. Let $r_0,m_0$ be the remainder and quotient for the Euclidean division of $m$ by $\Lambda_0$. More generally, we construct $m_{j+1},r_{j+1}$ as respectively the quotient and remainder of The Euclidean division of $m_j$ by $\Lambda_j$ for $j=0,\ldots,\Omega(N)-1$. Set:
\[
\eta_{n\Omega(N)+j} = \Lambda_j \eta_{n\Omega(N)+j+1} - r_j
\]
for all $j=0,\Omega(N)-1$. We have given two definitions of $\eta_{n\Omega(N)}$ and need to check they give the same values:
\begin{eqnarray*}
N \eta_{(n+1)\Omega(N)} &=& \Lambda_0 \cdots \Lambda_{\Omega(N)-1} \eta_{(n+1)\Omega(N)}\\
&=& \Lambda_0 \cdots \Lambda_{\Omega(N)-2} (\eta_{(n+1)\Omega(N)-1} + r_{\Omega(N)-1})\\
\cdots &=& \eta_{n\Omega}+r_0 + \Lambda_0(r_1 + \Lambda_1( r_2 +\cdots)) = \eta_{n\Omega(N)}+k
\end{eqnarray*}
so our construction leads to a coherent result. Now, by construction, $\eta\in\Xi_{\Lambda(N)}$, and $\omega(\eta)=\nu$. Hence $\omega$ is a group isomorphism. This completes our proof.
\end{proof}

\bigskip 
\begin{remark}
The group $\Xi_\Lambda$ can be topologized as a subspace of $([0,1]/\sim)^\N$ where $\sim$ is the equivalence relation defined by $x\sim y \iff (x=y)\vee (x=0\wedge y=1)$. With this topology, the natural isomorphism $e$ is of course an homeomorphism, so that $\Xi_\Lambda$ is isomorphic to:
\[
\solenoid_\Lambda = \{ (z_n)_{n\in\N} \in \T^{\N}: \forall n \in \N \;\; z_{n+1}^{\Lambda_n} = z_n \}
\]
as a topological group, though we shall not need this.
\end{remark}

\bigskip We are now ready to establish a necessary and sufficient condition for the symmetrizer group of a given multiplier to be nontrivial.

\begin{theorem}\label{symmetrizer}
Let $N\in\N,N>1$. Let $\alpha\in\Xi_N$. The symmetrizer subgroup in $\Q_N^2$ for $\Psi_\alpha$ is defined by: 
\[
\mathcal{S}_\alpha = \left\{ g = \left( \frac{p_1}{N^{k_1}},\frac{p_2}{N^{k_2}}\right) \in \Q_N^2 : 
\Psi_\alpha(g,\cdot) = \Psi(\cdot,g) \right\}\text{.}
\]
The following assertions are equivalent:
\begin{enumerate}
\item The symmetrizer group $\mathcal{S}_\alpha$ is non-trivial,
\item The sequence $\alpha$ has finite range (i.e. $\{\alpha_n : n\in\N \}$ is finite).
\item There exists $j<k\in\N$ such that $\alpha_j=\alpha_k$,
\item There exists $k\in\N$ such that $(N^k-1)\alpha_0 \in \Z$,
\item The sequence $\alpha$ is periodic.
\item The group $\mathcal{S}_\alpha$ is either $\Q_N^2$ (which is equivalent to $\alpha= 0$) or there exists a nonzero $b\in\N$ such that:
\[
\mathcal{S}_\alpha = \left\{ \left(\frac{p_1b}{N^m},\frac{p_2b}{N^n}\right) : p_1,p_2\in\Z,n.m\in\N \right\}\text{.}
\]
\end{enumerate}
\end{theorem}

\begin{proof}
Let us assume that $s_\alpha$ is nontrivial and prove that the range of $\alpha$ is finite. The result is trivial if $\alpha=(0)_{n\in\N}$, so we assume that there exists $s\in\N$ such that $\alpha_s\not=0$. By definition of $\Xi_\Lambda$, we then have $\alpha_n\not=0$ for all $n\geq s$.

Let $\Theta_\alpha: (x,y)\in\Q_N^2 \mapsto \Psi_\alpha(x,y)\Psi_\alpha(y,x)^{-1}$. Now, given $p_1,p_2,p_3,p_4\in\Z$ and $k_1,k_2,k_3,k_4\in\N$, we have $\Theta_\alpha\left(\left(\frac{p_1}{N^{k_1}},\frac{p_2}{N^{k_2}}\right),\left(\frac{p_3}{N^{k_3}},\frac{p_4}{N^{k_4}}\right)\right)$ given by:
\[
\exp\left(2i\pi\left(\alpha_{(k_1+k_4)}p_1p_4 - \alpha_{(k_2+k_3)}p_2p_3 \right)\right)\text{.}
\]
The symmetrizer group $s_\alpha$ is now given by:
\[
\left\{ g = \left( \frac{p_1}{N^{k_1}},\frac{p_2}{N^{k_2}}\right) \in \Q_N^2 : 
\Theta_\alpha(g,\cdot) = 1 \right\}\text{.}
\]

Fix $\left(\frac{n}{N^{k_1}},\frac{m}{N^{k_2}}\right)\in \mathcal{S}_\alpha$, so that for all $\left(\frac{p_3}{N^{k_3}},\frac{p_4}{N^{k_4}}\right)\in\Q_N^2$ we have:
\[
\Theta_\alpha\left(\left(\frac{n}{N^{k_1}},\frac{m}{N^{k_2}}\right), \left(\frac{p_3}{N^{k_3}},\frac{p_4}{N^{k_4}}\right)\right)=1\text{.}
\]
Then, by Theorem (\ref{multipliers}), for all $p_3,p_4\in\Z$ and $k_3,k_4\in\N$:
\begin{equation} \label{dioph_condition_0}
\alpha_{(k_1+k_4)}np_4 \equiv \alpha_{(k_2+k_3)}mp_3 \mod \Z \text{.}
\end{equation}
Since Congruence (\ref{dioph_condition_0}) only depends on $k_1+k_4$ and must be true for all $k_4\in\N$, we can and shall henceforth assume that $k_1\geq s$. Without loss of generality, we assume $n\not=0$ (if $n=0$, then $m\not=0$ and the following argument can be easily adapted).

Denote by $\beta$ the unique extension of $\alpha$ in $\Xi_{\Lambda(N)}$ and denote $\Lambda(N)$ simply by $\Lambda$. Congruence (\ref{dioph_condition_0}) implies that for all $k_3,k_4\in\N$:
\begin{equation}\label{dioph_condition_1}
\beta_{\Omega(N)(k_1+k_4)}np_4 \equiv \beta_{\Omega(N)(k_2+k_3)}mp_3 \mod \Z \text{.}
\end{equation}
Since $\prod_{l=j}^{r-1}  \beta_{j+r} \equiv \Lambda_{j} \mod \Z$ for all $j,r\in\N,r>0$, we conclude that for any $k_3,k_4\in\N$ we have:
\begin{equation}
\beta_{\Omega(N)(k_1)+k_4}np_4 \equiv \beta_{\Omega(N)(k_2)+k_3}mp_3 \mod \Z \text{,}
\end{equation}
or, more generally, for any $l_1\geq\Omega(N)k_1$, we have:
\begin{equation}\label{dioph_condition_2}
\beta_{l_1+k_4}np_4 \equiv \beta_{\Omega(N)(k_2)+k_3}mp_3 \mod \Z \text{,}
\end{equation}
for all $k_3,k_4\in\N$. We shall now modify $\Lambda$ and $\beta$ so that we may assume that $n$ in Congruence (\ref{dioph_condition_2}) may be chosen so that $n$ is relatively prime with $N$.

To do so, we write $n=n_1 Q$ with $n_1\in\Z$ relatively prime with $N$ and the set of prime factors of $Q\in\N$ is a subset of the set of prime factors of $N$. Let $k\in\N$ be the smallest integer such that $Q$ divides $\pi_{k\Omega(N)}(\Lambda)$ and $k\geq k_1$. Such a natural number exists by definition of $Q$ and $\Lambda$. Let $j_1<j_2<\cdots<j_r\in\N$ such that $j_r<\Omega(N)k$ and $Q=\prod_{l=1}^r \Lambda_{j_l}$: such a choice of integers $j_1,\ldots,j_r$ exists by definition of $k$. We also note that $r=\Omega(Q)-1$. Let $z_1<z_2<\cdots<z_t\in\N$ be chosen so that:
\[
\{z_1,\ldots,z_t,j_1,\ldots,j_r \} = \{ 0,\ldots,\Omega(N)k-1 \}\text{.}
\] 
We now define the following permutation of $\N$:
\[
s \left| \begin{array}{ccl}
\N & \longrightarrow & \N \\
x & \longmapsto & \left\{ \begin{array}{ccl}
\Omega(N)k - l & \text{if} & x = j_l\\
l &\text{if}& x = z_l\\
x & \multicolumn{2}{l}{\text{otherwise.}}
\end{array}\right.
\end{array}\right.
\]

Let $\Lambda'\in\primeseq$ be defined by $\Lambda'_j = \Lambda_{s(j)}$ for all $j\in\N$. By construction, $\Lambda$ and $\Lambda'$ agree for indices greater or equal than $\Omega(N)k$. Let $\alpha$ be the unique sequence in $\Xi_{\Lambda'}$ such that $\alpha'_{k\Omega(N)+j} = \beta_{k\Omega(N)+j}$ for all $j\in\N$. By construction, for all $k_3,k_4\in\N$, we have:
\begin{equation}\label{dioph_condition_3}
\alpha'_{\Omega(N)k+k_4}np_4 \equiv \beta_{\Omega(N)(k_2)+k_3}mp_3 \mod \Z \text{.}
\end{equation}
Yet $n=n_1Q$ and by construction, $\alpha'_{\Omega(N)k+k_4}Q\equiv \alpha'_{\Omega(N)k+k_4-r}n_1 \mod \Z$.

Thus, we have shows that if $\mathcal{S}_\alpha$ is not trivial, then there exists $\Lambda'\in\primeseq$ and a supersequence $\alpha'\in\Xi_{\Lambda'}$ of (a truncated subsequence of) $\alpha$, as well as $n_1\in\Z$ with the set of prime factors of $n_1$ disjoint from the range of $\Lambda'$ and $k,k_2\in\N$, such that for all $j,j'\in\N$ and $p,q\in\Z$, we have:
\begin{equation}\label{dioph_condition}
\alpha'_{k+j} n_1 p \equiv \alpha'_{k_2+j'} mq \mod \Z \text{.}
\end{equation}

We now set $q=0$. This relation can only be satisfied if $\alpha'_k \in \Q$, in which equivalent to $\alpha'_j \in \Q$ for all $j\in\Q$ by definition of $\Xi_{\Lambda'}$. 
Since Congruence (\ref{dioph_condition}) implies that $\alpha'_k n \in \Z$, we write $\alpha'_k = \frac{a}{b}$ with for some $b\in\Z$ such that $b\mid n_1$ and $b\wedge a = 1$, where $a\in\{1,\ldots,b-1\}$.

Now, by definition of $\Xi_{\Lambda'}$, there exists $x\in \{0,\ldots,\Lambda'_k-1 \}$ such that:
\[
\alpha'_{k+1} = \frac{\alpha'_k+x}{\Lambda'_k} = \frac{a + xb}{b\Lambda'_k}\text{.}
\]

We now must have:
\[
\alpha'_{k+1} n_1 = \frac{a+bx}{\Lambda'_k}\frac{n_1}{b} \in \Z
\]
which implies $\frac{a+bx}{\Lambda'_k}\in\N$ since $\Lambda'_k$ and $n_1$ are relatively prime. Hence we have:
\[
\alpha'_{k+1} \in \left\{ \frac{1}{b},\ldots,\frac{b-1}{b}\right\} \text{.}
\]

By induction, using the same argument as above, we thus get that we must have:
\begin{equation}\label{alpha_condition}
\{\alpha'_{k+j}:j\in\N\} \subseteq \left\{ \frac{1}{b},\ldots,\frac{b-1}{b}\right\}\text{.}
\end{equation}
 Hence if $\mathcal{S}_\alpha$ is nontrivial, then $\alpha'$ (and therefore $\alpha$) must have finite range. 

\begin{remark}\label{inproofremark}
Condition (\ref{alpha_condition}) implies that in fact, there exists $b,k\in\N$ such that for all $n\geq k$, there exists $a\in\{1,\ldots,b-1\}$ with $a\wedge b =1$ such that $\alpha'_n = \frac{a}{b}$. Indeed, since $\alpha'$ has finite range, there exists $K\in\N$ such that $\alpha'_m$ occurs infinitely often in $\alpha'$ for all $m>K$. Let $r=\max\{K,k\}$ and write $\alpha'_r = \frac{a}{b}$ for some $a,b\in\N$ with $a\wedge b =1$. if for any $n > r$, we have $\alpha'_n=\frac{a}{b'}$ with $a\wedge b'=1$ and $b'\mid b$, then Condition (\ref{alpha_condition}) implies that $b' \alpha'_m\in \Z$ for all $m>n$. By assumption on $r$, $\alpha'_r$ occurs again for some $r'>n$. Condition (\ref{alpha_condition}) then implies that $b\mid b'$, so $b=b'$.
\end{remark}

\bigskip It is obvious that if $\alpha$ has finite range, then there exists $j<k$ such that $\alpha_j = \alpha_k$.

\bigskip Let us now prove that if $\alpha$ takes the same value at least twice, then there exists $k\in\N$ such that $(N^k-1)\alpha_0 \in \Z$. Thus there exist  $j,k\in\N$ such that $\alpha_{j+k} = \alpha_j$, yet by definition of $\Xi_N$ we have $N^k \alpha_{j+k} \equiv \alpha_j \mod \Z$, so $(N^k-1)\alpha_j \equiv 0 \mod \Z$, and since $(N^k-1)\alpha_0 \equiv N^j (N^k-1)\alpha_j \mod\Z \equiv 0 \mod \Z$ , we conclude that $(N^k-1)\alpha_0 \in \Z$.

\bigskip Let us now assume that there exists $k\in\N$ such that $(N^k-1)\alpha_0\in\Z$ and show that $\alpha$ is periodic. If $\alpha_j=\frac{a}{b}$ for some $j\in\N$ and some $a,b\in\N$ nonzero and relatively prime, then $\alpha_{k+j}=\frac{d}{b}$ for some $d\in\{1,\ldots,b-1\}$ and:
\[
N^k\alpha_{k+j} = \frac{N^kd}{b} = \frac{(N^k-1)d}{b} + \frac{d}{b} \equiv \frac{d}{b}\mod \Z\text{,} 
\]
while we must have $N^k \alpha_{k+j}\equiv \frac{a}{b} \mod \Z$, which implies $d=a$. Hence by induction, $\alpha_{k+j}=\alpha_j$  for all $j\in\N$, as desired.

\bigskip Let us assume that $\alpha$ is periodic, which of course implies $\alpha_0=\frac{a}{b}$ for some relatively prime $a,b\in\Z$, or $\alpha=0$. In the former case, we simply have:
\[
\Psi_\alpha\left(\left(\frac{n}{N^{k_1}},\frac{m}{N^{k_2}}\right),\left(\frac{p_2}{N^{k_3}},\frac{q_2}{N^{k_4}}\right)\right)
 = \exp\left(\frac{2i\pi}{b} a_{k_1+k_4} nq_2\right)
\]
where $\alpha_j = \frac{a_j}{b}$ for $a_j\in\{1,\ldots,b-1\}$ and all $j\in\N$, using Remark (\ref{inproofremark}). The computation of $\mathcal{S}_\alpha$ is now trivial. It is also immediate, of course, if $\alpha=0$.
In particular, this computation shows that $\mathcal{S}_\alpha$ is not trivial if $\alpha$ is periodic, which concludes our equivalence.
\end{proof}

\begin{remark}
We note that if the symmetrizer group of the multiplier $\Psi_\alpha$ for $\alpha\in\Xi_N$ is nontrivial, then $\alpha$ is rational valued. The converse is false, as it is easy to construct an aperiodic $\alpha\in\Xi_N$ which is rational valued: for instance, given any $N>1$ we can set $\alpha_n = \frac{1}{N^n}$ for all $n\in\N$. Then $s_{\alpha}=\{0\}$.
\end{remark}

\begin{example}
For an example of a periodic multiplier, one can choose $N=5$ and $\alpha = \left(\frac{1}{62},\frac{25}{62},\frac{5}{62},\frac{1}{62},\ldots\right)$. The symmetrizer group is then given by $$\left\{ \left(\frac{62n}{5^p},\frac{62m}{5^q}\right): n,m\in\Z, p,q\in\N \right\}\text{.}$$
\end{example}

\section{The Noncommutative Solenoid $C^\ast$-algebras}

We now start the analysis of the noncommutative solenoids, defined by:

\begin{definition}
Let $N\in\N$ with $N>1$ and let $\alpha\in\Xi_N$. Let $\Psi_\alpha$ be the skew bicharacter defined in Theorem (\ref{multipliers}). The twisted group C*-algebra $C^\ast(\Q_N^2,\Psi_\alpha)$ is called a \emph{noncommutative solenoid} and is denoted by $\algebra_\alpha$.
\end{definition}

The main purpose of this and the next section is to provide a classification result for noncommutative solenoids based upon their defining multipliers. The key ingredient for this analysis is the computation of the $K$-theory of noncommutative solenoids, which will occupy most of this section. However, we start with a set of basic properties one can read about noncommutative solenoids from their defining multipliers.

It is useful to introduce the following notations, and provide an alternative description of our noncommutative solenoids.

\begin{notation}\label{Wunitaries}
Let $\alpha\in\Xi_N$ for some $N\in\N,N>1$. By definition, $\algebra_\alpha$ is the universal C*-algebra for the relations 
\[
W_{\frac{p_1}{N^{k_1}}, \frac{p_2}{N^{k_2}}} W_{\frac{p_3}{N^{k_3}}, \frac{p_4}{N^{k_4}}} = \Psi_\alpha\left(\left( \frac{p_1}{N^{k_1}}, \frac{p_2}{N^{k_2}} \right),\left( \frac{p_3}{N^{k_3}}, \frac{p_4}{N^{k_4}} \right)\right) W_{\frac{p_1}{N^{k_1}} +  \frac{p_3}{N^{k_3}}, \frac{p_2}{N^{k_2}} +  \frac{p_4}{N^{k_4}}}
\]
where $W_{x,y}$ are unitaries for all $(x,y)\in\Q_N^2$, and $p_1,p_2,p_3,p_4\in\Z$ and $k_1,k_2,k_3,k_4\in\N$. 
\end{notation}

\begin{proposition}\label{crossed-product}
Let $N\in\N,N>1$ and $\alpha\in\Xi_N$. Let $\theta^\alpha$ be the action of $\Q_N$ on $\solenoid_N$ defined by:
\[
\theta^\alpha_{\frac{p}{N^k}} \left((z_n)_{n\in\N} \right) = \left(\exp(2i\pi\alpha_{k+n} p) z_n \right)_{n\in\N}\text{.}
\]
The C*-crossed-product $C(\solenoid_N)\rtimes_{\theta^\alpha} \Q_N$ is *-isomorphic to $\algebra_\alpha$.
\end{proposition}

\begin{proof}
The C*-algebra $C(\solenoid_N)$ of continuous functions on $\solenoid_N$ is the group C*-algebra of the dual of $\solenoid_N$, i.e. it is generated by unitaries $U_p$ for $p\in\Q_N$ such that $U_p U_{p'} = U_{p+p'}$. Equivalently, it is the universal C*-algebra generated by unitaries $u_n$ such that $u_{n+1}^N=u^n$, with the natural *-isomorphism $\varphi$ extending $\left(\forall n \in \N \;\; u_n\mapsto U_{\frac{1}{N^n}}\right)$.

\bigskip The C*-crossed-product $C(\solenoid_N)\rtimes_{\theta^\alpha}\Q_N$ is generated by a copy of $C(\solenoid_N)$ and unitaries $V_q$, for $q\in\Q_N$, such that $V_q u_n V_q^\ast = \theta^\alpha_{\frac{1}{N^q}}\left(\frac{1}{N^n}\right) u_n$. Thus:
\begin{eqnarray*}
V_{\frac{p_1}{N^{k_1}}} U_{\frac{p_2}{N^{k_2}}} &=& \theta^\alpha_{\frac{p_1}{N^{k_1}}}\left(\frac{p_2}{N^{k_2}}\right) U_{\frac{p_2}{N^{k_2}}} V_{\frac{p_1}{N^{k_1}}}\\
 &=& \exp(2i\pi\alpha_{\delta(N^{k_1})+\delta(N^{k_2})}p_1p_2)\left(\frac{p_2}{N^{k_2}}\right) U_{\frac{p_2}{N^{k_2}}} V_{\frac{p_1}{N^{k_1}}}
\end{eqnarray*}
for all $p_1,p_2\in\Z$ and $k_1,k_2\in\N$. 
\bigskip Now, the following map (using Notation (\ref{Wunitaries})):
\[
\forall p\in\Z, k\in\N\;\;
\left\{
\begin{array}{lcr}
U_{\frac{p}{N^k}} &\longmapsto& W_{0,\frac{p}{N^k}}\\
V_{\frac{p}{N^k}} &\longmapsto& W_{\frac{p}{N^k},0}
\end{array}
\right.
\]
can be extended into a *-epimorphism using the universal property of $C(\solenoid_N)\rtimes_{\theta^\alpha} \Q_N$. The universal property of $\algebra_\alpha$ implies that this *-morphism is a *-isomorphism, by showing the inverse of this *-epimorphism is a well-defined *-epimorphism.
\end{proof}

\begin{remark}
Let $N\in\N,N>1$ and $\alpha\in\Xi_N$. The action $\theta$ of $\Q_N$ on $\solenoid_N$ defined in Proposition (\ref{crossed-product}) is minimal if and only if $\alpha$ is irrational-valued. However, if $\alpha$ has infinite range, the orbit space of $\theta$ is still a single topological point.
\end{remark}

\bigskip We start our study of noncommutative solenoids by establishing when these C*-algebras are simple:

\begin{theorem}\label{simplicity}
Let $N\in\N$ with $N>1$. Let $\alpha\in\Xi_N$. The following statements are equivalent:
\begin{enumerate}
\item The C*-algebra $\algebra_\alpha$ is simple,
\item The set $\{\alpha_n : n \in \N\}$ is infinite,
\item For all $k\in\N$ with $k>0$, we have $(N^k-1)\alpha_0\not\in\Z$,
\item Given any $j,k\in\N$ with $j\not=k$ we have $\alpha_j \not= \alpha_k$.
\end{enumerate}
\end{theorem}

\begin{proof}
The symmetrizer group $\mathcal{S}_\alpha$ of $\Psi_\alpha$ is trivial if and only if the asserted condition holds, by Theorem (\ref{symmetrizer}). Since $\Q_\Lambda^2$ is Abelian, and since the dual of $\mathcal{S}_\alpha$ is trivial, the action of $\Q_\Lambda^2/Q_\Lambda^2$ on $\widehat{\mathcal{S}_\alpha}$ is free and minimal. Thus $\algebra_\alpha$ is simple by \cite[Theorem 1.5]{Packer92}.
\end{proof}

As our next observation, we note that noncommutative solenoids carry a trace, which will be a useful tool for their classification.

\begin{theorem}\label{traces}
Let $N\in\N,N>1$ and $\alpha\in\Xi_N$. The C*-algebra $\algebra_\alpha$ has an invariant tracial state for the dual action of $\solenoid_N^2$. Moreover, if $\algebra_\alpha$ is simple, then this is the only tracial state of $\algebra_\alpha$.
\end{theorem}

\begin{proof}
\bigskip For any $\alpha\in\Xi_N$ for $N\in\N,N>1$, the group $\solenoid_N^2$ acts ergodically and strongly continuously on $\algebra_\alpha$ by setting, for all $(z,w)\in\solenoid_N$ and $(x,y)\in\Q_N^2$:
\[
(z,w)\cdot W_{x,y} = \left< z,x \right>\left<w,y\right> W_{x,y}
\]
and extending $\cdot$ by universality of $\algebra_\alpha$, using Notation (\ref{Wunitaries}). This is of course the dual action of $\solenoid_N^2$ on $C^\ast(\Q_N^2,\Psi_\alpha)$. Since $\solenoid_N^2$ is compact, the existence of an invariant tracial state $\tau$ is due to \cite{Hoegh-Krohn81}. Moreover, $\algebra_\alpha$ is simple if and only if $\Psi_\alpha^2(g,\cdot)=1$ only for $g=0$, by Theorem (\ref{simplicity}). If $\tau'$ is any tracial state on $\algebra_\alpha$, we must have (using Notation (\ref{Wunitaries})):
\[
\tau'(W_g W_h) = \Psi_\alpha^2(g,h) \tau'(W_h W_g)
\]
for all $g,h\in\Q_N^2$. Hence if $\algebra_\alpha$ is simple, we have $\tau(W_g W_h) = 0$ for all $g,h\in\Q_N^2$, except for $h\in\{g,g^{-1}\}$. So $\ker\tau=\ker\tau'$ and $\tau(1)=1=\tau'(1)$, so $\tau=\tau'$ as desired.
\end{proof}

As our next observation, the C*-algebras $\algebra_\alpha$ ($\alpha\in\Xi_N,N\in\N,N>1$) are inductive limit of rotation algebras. Rotation C*-algebras have been extensively studied, with \cite{Rieffel81,Elliott93b} being a very incomplete list of references. We recall that given $\theta\in[0,1)$, the rotation C*-algebra $A_\theta$ is the universal C*-algebra for the relation $VU=\exp(2i\pi\theta)UV$ with $U,V$ unitaries. It is the twisted group C*-algebra $C^\ast(\Z^2,\Theta)$ where $\Theta((n,m),(p,q))=\exp(i\pi\theta(nq-mp))$. The unitaries associated to $(1,0)$ and $(0,1)$ in $C^\ast(\Z^2,\Theta)$ will be denoted by $U_\theta$ and $V_\theta$ and referred to as the canonical unitaries of $A_\theta$. Of course, $\{U_\theta,V_\theta\}$ is a minimal generating set of $A_\theta$. We now have:

\begin{theorem}\label{ATorus}
Let $N\in\N$ with $N>1$ and $\alpha\in\Xi_\N$. For all $n\in\N$, let $\varphi_n$ be the unique *-morphism from $A_{\alpha_{2n}}$ into $A_{\alpha_{2n+2}}$ extending:
\[
\left\{
\begin{array}{lcr}
U_{\alpha_{2n}} &\longmapsto& U_{\alpha_{2n+2}}^N\\
V_{\alpha_{2n}} &\longmapsto& V_{\alpha_{2n+2}}^N\\
\end{array}
\right.
\]
Then:
\[
\begin{CD}
A_{\alpha_0} @>\varphi_0>> A_{\alpha_2} @>\varphi_1>> A_{\alpha_4} @>\varphi_2>> \cdots
\end{CD}
\]
converges to $\algebra_\alpha$, where $A_\theta$ is the rotation C*-algebra for the rotation of angle $2i\pi\theta$.
\end{theorem}

\begin{proof}
We use Notations (\ref{Wunitaries}). Consider the given sequence of irrational C*-algebra. Fix $k\in\N$. Define the map:
\[
\upsilon_k : \left\{ \begin{array}{ccc}

U_{\alpha_{2k}} &\mapsto& W_{\frac{1}{N^k},0}\\
V_{\alpha_{2k}} &\mapsto& W_{0,\frac{1}{N^k}}

\end{array}\right.
\]

By definition of $\Psi_\alpha$, we have $W_{0,\frac{1}{N^k}}W_{\frac{1}{N^k},0} = e^{2i\pi\alpha_{2k}} W_{\frac{1}{N^k},0}  W_{0,\frac{1}{N^k}}$.

By universality of $A_{\alpha_{2k}}$, the map $\upsilon_k$ extends to a unique *-morphism, which we still denote $\upsilon_k$, from $A_{\alpha_{2k}}$ into $\algebra_\alpha$. It is straightforward to check that the diagram:
\[
\begin{CD}
A_{\alpha_0} @>\varphi_0 >> A_{\alpha_2} @>\varphi_1>> A_{\alpha_4} @>\varphi_2>> \cdots\\
 @VV\upsilon_0 V @VV\upsilon_1V @VV\upsilon_2V \cdots\\
\algebra_\alpha @= \algebra_\alpha @= \algebra_\alpha @= \cdots
\end{CD}
\]
commute. So by universality of the inductive limit, there is a morphism from $\varinjlim (A_{\alpha_{2k}},\varphi_k)_{k\in\N}$ to $\algebra_\alpha$. Now, since $\algebra_\alpha$ is in fact generated by $\bigcup_{k\in\N}\upsilon_k(A_{\alpha_{2k}})$, we conclude that $\algebra_\alpha$ is in fact $\varinjlim (A_{\alpha_{2k}},\varphi_k)_{k\in\N}$, as desired.
\end{proof}

We can use Theorem (\ref{ATorus}) to compute the $K$-theory of the C*-algebras $\algebra_\alpha$ for $N\in\N,\alpha\in\Xi_N$.

\begin{theorem}\label{Ktheory}
Let $N\in\N$ with $N>1$, and let $\alpha\in\Xi_N$. Define the subgroup $\mathscr{K}_\alpha$ of $\Q_N^2$ by:
\[
\mathscr{K}_\alpha = \left\{ \left( z + \frac{pJ_k^\alpha}{N^{k}}, \frac{p}{N^{k}} \right) :z,p \in \Z, k \in \N \right\} 
\]
where $(J_k^\alpha)_{k\in\N} = (N^{k} \alpha_{k} - \alpha_0)_{k\in\N}$ and by convention, $J_k = 0$ for $k\leq 0$.
We then have:
\begin{equation*}
K_0(\algebra_\alpha)  = \mathscr{K}_\alpha \text{,} \; \text{and}\;
K_1(\algebra_\alpha) = \Q_N^2 \text{.}
\end{equation*}
Moreover, if $\tau$ is a tracial state of $\algebra_\alpha$, then we have:
\begin{equation}\label{K-traces}
K_0(\tau) : \left(z+\frac{pJ^\alpha_k}{N^{k}} ,\frac{p}{N^{k}}\right) \in \mathscr{K}_\alpha \longmapsto z + p\alpha_{k}\text{.}
\end{equation}
In particular, all tracial states of $\algebra_\alpha$ lift to the same state of $K_0(\algebra_\alpha)$ given by (\ref{K-traces}).
\end{theorem}

\begin{proof}
Define $j_n^\alpha\in\{0,\ldots,N-1\}$ for $n\in\N$ by $N\alpha_{n+1} = \alpha_{n} + j^\alpha_n$, so that by definition:
\[
J^\alpha_{n} = \sum_{k=0}^{n-1} N^{k}j^\alpha_k\text{.}
\]
To ease notations, we also introduce for all $n\in\N$ the integer $r^\alpha_n \in \{0,\ldots,N^2-1\}$ such that $N^2\alpha_{2n+2} = \alpha_{2n} + r^\alpha_n$. Thus $r^\alpha_n = Nj^\alpha_{2n+1} + j^\alpha_{2n}$ and $J^\alpha_{2n} = \sum_{k=0}^{n-1} N^{2k}r^\alpha_k$ for all $n\in\N$.

As a preliminary step, we check that $\mathscr{K}_\alpha$ is a group. It is a nonempty subset of $\Q_N^2$ since it contains $(0,0)$. Now, let $\left(z+\frac{pJ_k^\alpha}{N^{k}}, \frac{p}{J_k^\alpha}\right)$ and $\left(y+\frac{qJ_r^\alpha}{N^r},\frac{q}{N^r}\right)$ be elements of $\mathscr{K}_\alpha$. Let $n = \max(k,r)$, and $m_1,m_2 \in \N$ be given so that $N^k m_1 = N^n$ and $N^r m_2 = N^n$. We then have:
\begin{equation}\label{Kgroupproof}
\left(z+\frac{pJ_k^\alpha}{N^k}, \frac{p}{N^k}\right) - \left(y+\frac{qJ_r^\alpha}{N^r}, \frac{q}{N^r}\right) = \left(z-y+\frac{ m_1pJ_k^\alpha - m_2qJ_r^\alpha}{N^n}, \frac{m_1p - m_2q}{N^n}\right)\text{.}
\end{equation}
Now, assume $k<n$, so $r=n$. By definition, $J_n^\alpha = J_k^\alpha + N^k j_k^\alpha + \cdots + N^{n-1}j_{n-1}^\alpha$ so $m_1 J_k^\alpha = m_1 J_n^\alpha - (N^nj_k^\alpha + N^{n+1}j_{k+1}^\alpha +\cdots + N^{2n-1}j_{n-1}^\alpha)$, so $\frac{m_1pJ_k^\alpha}{N^n} = - j_k^\alpha - \cdots - N^{n-1}j_{n-1}^\alpha + \frac{m_1pJ_n^\alpha}{N^n}$. In this case, Expression (\ref{Kgroupproof}) becomes:
\[
\left(z-y-j_k^\alpha-\cdots-N^{n-1}j_{n-1}^\alpha+\frac{ (m_1p- m_2q)J_n^\alpha}{N^n}, \frac{m_1p - m_2q}{N^n}\right)
\]
which lies in $\mathscr{K}_\alpha$. The computations are similar if we assume instead $r<n$ and $k=n$. Thus $\mathscr{K}_\alpha$ is a subgroup of $\Q_N^2$. We remark here that the sequences $(j_k)_{k\in\N}$ and $(J_k)_{k\in\N}$ are closely related to the group of $N$-adic integers $\Z_N$; we shall discuss this relationship in detail at the conclusion of the proof.

\bigskip We simplify our notations in this proof and denote the canonical unitaries of the rotation C*-algebra $A_{\alpha_{2k}}$ as $U_k$ and $V_k$ for all $k\in\N$. It is well known that:
\[
K_0(A_{\alpha_{2k}})  = \Z^2 \;\;\text{and}\;\;
K_1(A_{\alpha_{2k}})  = \Z^2 \text{.}
\]
Moreover, $K_0(A_{\alpha_{2k}})$ is generated by the classes of the identity and a Rieffel projection $P$ of trace $\alpha_{2k}$, which we denote by $(1,0)$ and $(0,1)$ respectively. We also know that $K_1(A_{\alpha_{2k}})$ is generated by the classes of $U_k$ and $V_k$, denoted respectively by $(1,0)$ and $(0,1)$.

We start with a key observation. Let $P$ be a Rieffel projection of trace $\alpha_{2k}$ in $A_{\alpha_{2k}}$, then it is of the form $g(U_k)V_k+f(U_k)+h(U_k)V_k^\ast$ with $f,g,h\in C(\T)$ and $\alpha_{2k} = \int_{\T} f$. Hence $P$ is mapped by $\varphi_k$ to the Rieffel projection $g(U_{k+1}^{N})V_{k+1}^{N} + f(U_{k+1}^{N}) + h(U_{k+1}^{N})V_{k+1}^{N}$ whose trace is again $\alpha_{2k}$. We recall that with our notation:
\[
N^2 \alpha_{2k+2} = \alpha_{2k} + r_{k}^\alpha \text{,}
\]
where we note that $\alpha_{2k+2}$ is the trace of the generator of $K_0(A_{2k+2})$.
Let $k\in\N$ and let $\varphi_k$ be the *-morphism defined in Theorem (\ref{ATorus}). The maps $K_0(\varphi_k)$ and $K_1(\varphi_k)$ are thus completely determined, as morphisms of $\Z^2$, by the relations:
\[
K_1(\varphi_k) : \left\{ \begin{array}{ccc}
(1,0) &\mapsto& (N,0)\\
(0,1) &\mapsto& (0,N)
\end{array}\right.
\;\text{and}\;
K_0(\varphi_k) :  \left\{ \begin{array}{ccc}
(1,0) &\mapsto& (1,0)\\
(0,1) &\mapsto& (r_{2k}^\alpha,N)
\end{array}\right.
\]

 We now use the  continuity of $K$-theory groups to conclude:
\begin{eqnarray*}
K_1(\algebra_\alpha) &=& \varinjlim\;\left(
\xymatrix@1{
{\Z^2} \ar[rr]^{K_1(\varphi_0)} & & {\Z^2} \ar[rr]^{K_1(\varphi_2)} & & {\Z^2} \ar[rr]^{K_1(\varphi_2)} & &\cdots
}\right)\\
&=& \Q_N^2\text{,}
\end{eqnarray*}
and
\begin{eqnarray*}
K_0(\algebra_\alpha) &=& \varinjlim\;\left(
\xymatrix@1{
{\Z^2} \ar[rr]^{K_0(\varphi_0)} & & {\Z^2} \ar[rr]^{K_0(\varphi_2)} & & {\Z^2} \ar[rr]^{K_0(\varphi_2)} & &\cdots
}\right)\\
&=&\varinjlim\; \left(\xymatrix@1{
{\Z^2} \ar[rr]^{\small \left[ \begin{array}{cc} 1 & r_{0}^\alpha \\ 0 & N^2 \end{array} \right]} & & {\Z^2} \ar[rr]^{\small \left[ \begin{array}{cc} 1 & r_{1}^\alpha \\ 0 & N^2 \end{array} \right]} & & {\Z^2} \ar[rr]^{\small \left[ \begin{array}{cc} 1 & r_{2}^\alpha \\ 0 & N^2 \end{array} \right]} & &\cdots
}\right)
\text{.}
\end{eqnarray*}
We claim that the group $K_0(\algebra_\alpha)$ is $\mathscr{K}_\alpha$. For $k\in\N$ we define $\upsilon_k : \Z^2 \rightarrow \mathscr{K}$ to be the multiplication by the matrix: 
\[
\left[ \begin{array}{cc} 1 & -\frac{J_k^\alpha}{N^{2k}} \\ 0 & \frac{1}{N^{2k}} \end{array} \right] = \prod_{n=0}^k \left[ \begin{array}{cc} 1 & r_{n-k}^\alpha \\ 0 & N^2 \end{array} \right]^{-1}\text{.}
\]

We now check the following diagram is commutative:
\[
\xymatrix{
{\Z^2} \ar[rr]^{\small \left[ \begin{array}{cc} 1 & r_{0}^\alpha \\ 0 & N^2 \end{array} \right]} \ar[d]^{\upsilon_0} & & {\Z^2} \ar[rr]^{\small \left[ \begin{array}{cc} 1 & r_{1}^\alpha \\ 0 & N^2 \end{array} \right]} \ar[d]^{\upsilon_1} & & {\Z^2} \ar[rr]^{\small \left[ \begin{array}{cc} 1 & r_{2}^\alpha \\ 0 & N^2 \end{array} \right]} \ar[d]^{\upsilon_2} & &\cdots
\\
\mathscr{K}_\alpha \ar@{=}[rr] & & \mathscr{K}_\alpha \ar@{=}[rr] & & \mathscr{K}_\alpha \ar@{=}[rr] & & \cdots
}
\]
It is now easy to check that $\mathscr{K}$ is indeed $K_0(\algebra_\alpha)$.

\bigskip Let $\tau$ be a tracial state of $\algebra_\alpha$. First, we note that $(1,0)\in\mathscr{K}$ is the image of $(1,0)\in\Z^2$ for all $\upsilon_k$, with $k\in\N$. Since $\tau(1)=1$ in $A_{2k}$ for all $k\in\N$, we conclude that $K_0(\tau)(1,0) = 1$. On the other hand, the element $\left(\frac{J_{2k}^\alpha}{N^{2k}},\frac{1}{N^{2k}}\right)$ is the image of $(0,1)\in\Z^2$ by $\upsilon_k$. The generator $(0,1)$ of $K_0(A_{\alpha_{2k}})$ has trace $\alpha_{2k}$, so $K_0(\tau)\left(\frac{J_{2k}^\alpha}{N^{2k}},\frac{1}{N^{2k}}\right) = \alpha_{2k}$ for all $k\in\N$.  
Now, since:
\[
\left(\frac{J_{2k-1}}{N^{2k-1}},\frac{1}{N^{2k-1}}\right) =
\left(-j^\alpha_{2k-1}+\frac{J_{2k}N}{N^{2k-1}},\frac{N}{N^{2k-1}}\right)
\]
and since $K_0(\tau)$ is a group morphism, we get:
\[
K_0(\tau) \left(\frac{J_{2k-1}}{N^{2k-1}},\frac{1}{N^{2k-1}}\right) =
-j_{2k-1}^\alpha + N\alpha_{2k} = \alpha_{2k-1}
\]
for all $k\in\N,k>1$. In summary, $K_0(\tau)$ maps $\left(\frac{J_{k}}{N^{k}},\frac{1}{N^{k}}\right)$ to $\alpha_{k}$ for all $k\in\N$. Using the morphism property of $K_0(\tau)$ again, we obtain the desired formula.
\end{proof}

The group $\mathscr{K}_\alpha$ defined in Theorem (\ref{Ktheory}) is in fact an extension of $\Q_N$ given by:
\begin{equation}\label{Kextension}
\begin{CD}
0  @>>> {\Z} @>\iota>> {\mathscr{K}_\alpha} @>\pi>> {\Q_N} @>>> 0\text{,}
\end{CD}
\end{equation}
where $\iota:z\in\Z\mapsto (z,0)$ is the canonical injection and $\pi:\left(z+\frac{pJ_k^\alpha}{N^k},\frac{p}{N^k}\right)\mapsto\frac{p}{N^k}$ is easily checked to be a group morphism such that the above sequence is exact. The class of this extension in $H^2(\Q_N,\Z)$ is however not in general an invariant of the *-isomorphism problem for noncommutative solenoids: as we shall explain in the next section, we must consider a weaker form of equivalence for Abelian extensions to construct such an invariant. It will translate into an equivalence relation on $\mathrm{Ext}(\Q_N,\Z_N)$ to be detailed after Theorem (\ref{classification}).

\bigskip We now proceed to provide a description of the $\Z$-valued $2$-cocycle of $\Q_N$ associated to Extension (\ref{Kextension}) and provide a different, more standard picture for $\mathscr{K}_\alpha$. Remarkably, we shall see that every element of $\mathrm{Ext}(\Q_N,\Z)$ is given by the $K$-theory of $\algebra_\alpha$ for some $\alpha\in\Xi_N$. As a first indication of this connection, we note that for a given $\alpha\in\Xi_N$, the sequence $(J^\alpha_k)_{k\in\N}$ can be seen an element of the group $\Z_N$ of $N$-adic integers \cite{Kaplansky69}. For our purpose, we choose the following description of $\Z_N$:

\begin{definition}[\cite{Kaplansky69}]\label{NadicInteger}
Let $N\in\N,N>1$. Set:
\[
\Z_N = \left\{ (J_k)_{k\in\N} : \wedge \left\{ \begin{array}{l}J_0=0, \\ \forall k \in \N \;\; \exists j\in\{0,\ldots,N-1\}\;\;\; J_{k+1} = J_k + N^k j\end{array}\right. \right\}\text{.}
\]
This set is made into a group with the following operation. If $J,K \in \Z_N$ then $J+K$ is the sequence $(L_k)_{k\in\N}$ where $L_k$ is the remainder of the Euclidean division of $J_k+K_k$ by $N^k$ for all $k\in\N$. This group is the group of $N$-adic integers.
\end{definition}

The connection between $\Xi_N$ (or equivalently, $\solenoid_N$), $\Z_N$, $\mathrm{Ext}(\Q_N,\Z)$ and $K_0$ groups of noncommutative solenoids is the matter of the next few theorems. We start by observing that the following is a short exact sequence:
\[
\begin{CD}
0 @>>> \Z_N @>\iota>> \Xi_N @>q>> \T @>>> 0
\end{CD}
\]
where $q:\alpha\in\Xi_N \mapsto \exp(2i\pi \alpha_0)$ and $\iota$ is the natural inclusion given by:
\[
\iota : (J_n)_{n\in\N} \in \Z_N \longmapsto \left(\frac{J_n}{N^n}\right)_{n\in\N} \text{.}
\]
Thus, for any element $\alpha$ of $\Xi_N$, the sequence $(J^\alpha_k)_{k\in\N}$ of Theorem (\ref{Ktheory}) associated to $\alpha$ is easily checked to be the unique element in $\Z_N$ such that $\alpha_k = q(\alpha) + J^\alpha_k$ for all $k\in\N$.

We shall use the following terminology:

\begin{definition}
Let $N\in\N,N>1$. The $N$-reduced form of $q\in\Q_N$ is $(p,N^k)\in\Z\times\N$ such that $q = \frac{p}{N^k}$ where $k$ is the smallest element of $\{ n \in \N : \exists p \in \Z\;q=\frac{q}{N^n}\}$. By standard abuse of terminology, we say that $\frac{p}{N^k}$ is $q$ written in its reduced form.
\end{definition}
A fraction in $N$-reduced form in $\Q_N$ may not be irreducible in $\Q$, so this notion depends on our choice of $N$. Namely, even if $\Q_N=\Q_M$ for $N\not=M$, and $\frac{p}{N^k}\in\Q_N$ is in $N$-reduced form, it may not be in $M$-reduced form. We shall however drop the prefix $N$ when the context allows it without introducing any confusion.

\bigskip We now prove the following lemma:

\begin{lemma}\label{QgroupCocycle}
Let $N\in\N,N>1$ and $\alpha\in\Xi_N$. Let $J=(J_k)_{k\in\N}\in\Z_N$.  Writing all elements of $\Q_N$ in their $N$-reduced form only, the map:
\[
\xi_J \left| \begin{array}{ccl}
\Q_N\times \Q_N & \longmapsto & \Z \\
\left(\frac{p_1}{N^{k_1}},\frac{p_2}{N^{k_2}} \right) & \longmapsto & 
\left\{
\begin{array}{lcl}

-\frac{p_1}{N^{k_1}}\left(J_{k_2}-J_{k_1}\right)&\text{if}&k_2>k_1\\
-\frac{p_2}{N^{k_2}}\left(J_{k_1}-J_{k_2}\right)&\text{if}&k_1>k_2\\
\frac{q}{N^r}\left(J_{k_1}-J_{r}\right)&\text{if}&\wedge\left\{\begin{array}{l}k_1=k_2\\ \frac{p_1}{N^{k_1}}+\frac{p_2}{N^{k_2}}=\frac{q}{N^r}\end{array}\right.\\
\end{array}\right.
\end{array}
\right.
\]
is a $\Z$-valued symmetric $2$-cocycle of $\Q_N$.
\end{lemma}

\begin{proof}
We introduce some useful notations for this proof. We defined $j_k\in\{0,\ldots,N-1\}$ for all $k\in\N$ by:
\[
J_{k+1} - J_{k} = N^k j_k\text{.}
\]
We also define $J_{k,m}$ for all $m,k\in\N, m>k$ by:
\[
J_{k,k} = 0 \wedge J_{k,m} = \sum_{r=k}^{m-1} N^{r-k} j_r\text{.}
\]
Note that $\frac{J_k-J_r}{N^r} = J_{r,k}$ for all $r\leq k$ by definition.

With this definition, we have $\xi_J\left(\frac{p_1}{N^{k_1}},\frac{p_2}{N^{k_2}}\right)$ equal to $-p_1J_{k_1,k_2}$ if $k_1<k_2$, to $-p_2J_{k_2,k_1}$ when $k_2<k_1$ and $qJ_{r,{k_1}}$ if $k_1=k_2$ and $p_1+p_2 = N^{k_1-r}q$, with $q$ and $N$ relatively prime, and with all fractions written in their reduced form in $\Q_N$.

\bigskip By construction, $\xi_J$ is a symmetric function. Let $x,y,z\in\Q_N$. We wish to show that:
\begin{equation}\label{Zcocycle}
\xi_J(x+y,z)+\xi_J(x,y) = \xi_J(y+z,x)+\xi_J(y,z)\text{.}
\end{equation}
Let us write $x=\frac{p_x}{N^{k_x}}$ in its reduced form, and use similar notations for $y$ and $z$. We proceed by checking various cases.
\begin{case}
Assume $x,y,z$ have the same denominator $N^k$ in their reduced form, and that $x+y=\frac{q}{N^r}$ in its reduced form, with $r<k$. Then by definition, $\xi_J(x,y)=qJ_{r,k}$ and $\xi_J(x+y,z)=-qJ_{r,k}$ so the left hand side of Identity (\ref{Zcocycle}) is zero.
Let $y+z=\frac{q'}{N^n}$ in its reduced form. If, again, $n<k$, the right hand side of Identity (\ref{Zcocycle}) is zero again and we have shown that Identity (\ref{Zcocycle}) holds. If $n=k$ then $\xi_J(y,z)=0$ by definition. Moreover, $x+y+z$ must have denominator $N^k$ in its reduced form. Indeed, since $x,y$ have the same denominator $N^k$ in reduced form, yet their sum does not, $p_x+p_y$ is a multiple of $N$. If moreover, $p_x+p_y+p_z$ is also a multiple of $N$, then $p_z$ is a multiple of $N$, which contradicts the definition of reduced form. Hence, $x+y+z$ has denominator $N^k$ in its reduced form and $\xi_J(x,y+z)=0$ by definition.
\end{case}

\begin{case}
Assume now that $k_x>k_y>k_z$. Then by definition:
\begin{equation}\label{Case2Eq}
\xi_J(x,y)+\xi_J(x+y,z)=-p_yJ_{k_y,k_x} - p_zJ_{k_z,k_x}
\end{equation}
while
\begin{equation}\label{Case2Eq2}
\xi_J(y,z)+\xi_J(y+z,x)=-p_zJ_{k_z,k_y} - (N^{k_y-k_z}p_z+p_y)J_{k_y,k_x}\text{.}
\end{equation}
By definition, $J_{k_z,k_y} + N^{k_y-k_z} J_{k_y,k_x} = J_{k_z,k_x}$. We then easily check that the left and right hand side of Identity (\ref{Zcocycle}) which are given by Identities (\ref{Case2Eq}) and (\ref{Case2Eq2}) agree. 

This case also handles the situation $k_z>k_y>k_z$ by switching the left and right hand side of Identity (\ref{Zcocycle}). 
\end{case}

\begin{case}
Assume now that $k_y>k_x>k_z$. Then the left hand side of  Identity (\ref{Zcocycle}) is given by:
\[
\xi_J(x,y)+\xi_J(x+y,z)=-p_xJ_{k_x,k_y} - p_zJ_{k_z,k_y}
\]
 On the other hand, the right hand side becomes:
\[
\xi_J(y,z)+\xi_J(y+z,x) = -p_zJ_{k_z,k_y} - p_x J_{k_x,k_y}
\]
and thus Identity (\ref{Zcocycle}) is satisfied again. We also get by symmetry the case $k_y>k_x>k_z$.
\end{case}

\begin{case}
Assume $k_x>k_z>k_y$. Then the left hand side of Identity (\ref{Zcocycle}) is:
\[
\xi_J(x,y)+\xi_J(x+y,z)=-p_yJ_{k_y,k_x} -p_z J_{k_z,k_x}
\]
while the right hand side is:
\[
\xi_J(y,z)+\xi_J(y+z,x) =-p_yJ_{k_y,k_z} - (N^{k_z-k_y}p_y+p_z)J_{k_z,k_x}
\]
and as in Case 1, both side agree. The last possible strict inequality $k_z>k_x>k_y$ is handle by symmetry again.
\end{case}

One similarly verifies that $\xi_J$ is a cocycle for the cases $k_y>k_x>k_z$, $k_x>k_z>k_y$, $k_x=k_y>k_z$ and $k_x=k_y>k_z$.

\end{proof}

\begin{theorem}\label{Qgroup}
Let $N\in\N, N>1$ and $\alpha\in\Xi_N$. Let $\xi_\alpha$ be the $\Z$-valued $2$-cocycle of $\Q_N$ given by $\xi_{J^\alpha}$ as defined in Lemma (\ref{QgroupCocycle}), where $J^\alpha_k = N^{k}\alpha_{k} - \alpha_0$ for all $k\in\N$.

Let us define the group $\mathscr{Q}_\alpha$ as the set $\Z\times\Q_N$ together with the operation:
\[
\left(z,\frac{p_1}{N^{k_1}}\right) \boxplus \left(y,\frac{p_2}{N^{k_2}}\right) = \left(z+y+\xi_\alpha\left(\frac{p_1}{N^{k_1}},\frac{p_2}{N^{k_2}}\right), \frac{p_1}{N^{k_1}}+\frac{p_2}{N^{k_2}}\right)
\]
for all $z,y,p_1,p_2\in\Z,k_1,k_2\in\N$. The map:
\[
\omega \left| \begin{array}{ccc}

\mathscr{Q}_\alpha & \longrightarrow & \mathscr{K}_\alpha\\
\left(z,\frac{p}{N^k}\right) &\longmapsto& \left(z+\frac{pJ^\alpha_k}{N^k},\frac{p}{N^k}\right)\text{.}

\end{array}
\right.
\]
is a group isomorphism. Thus $K_0(\algebra_\alpha)$ is isomorphic to $\mathscr{Q}_\alpha$ and, using $\omega$ to identify these groups, we have:
\[
K_0(\tau) : (1,0)\mapsto 1, \;\; \left(0,\frac{1}{N^k}\right)\mapsto \alpha_{k}\text{.}
\]
\end{theorem}

\begin{proof}
It is immediate that $\omega$ is a bijection. It remains to show that it is a group morphism. Let $x=\frac{p_x}{N^{k_x}}$, $y=\frac{p_y}{N^{k_y}}$ with $p_x,p_y\in\Z,k_x,k_y\in\N$. Let $z,t\in\Z$. We consider three distinct cases.
\begin{case}
The easiest case is when $k_x=k_y$ and $p_x+p_y$ is not a multiple of $N$. Then $\xi_\alpha(x,y)=0$ so $\boxplus$ reduces to the usual addition and we have:
\[
\omega(z,x)+\omega(t,y) = (z+x,x)+(t+y,y) = \omega(z+t,x+y) = \omega((z,x)\boxplus(t,y))\text{,}
\] 
as needed.
\end{case}

\begin{case}
Now, assume $k_x=k_y$ yet $p_x+p_y=N^{k-r} q$ for some $q$ not divisible by $N$ and some $r\in\N, r>0$. Then:
\begin{eqnarray*}
\omega(z,x)+\omega(t,y) &=& \left(z+t+\frac{(p_x+p_y)J^\alpha_{k_x}}{N^{k_x}},\frac{p_x+p_y}{N^{k_x}}\right)\\
&=&\left(z+t+\frac{qJ_{k_x}^\alpha}{N^r},\frac{q}{N^r}\right)\text{.}
\end{eqnarray*}
Now, $J^\alpha_{k_x} = J^\alpha_r + N^r J^\alpha_{r,k_x}$ by definition, as given in Theorem (\ref{Ktheory}) and Lemma (\ref{QgroupCocycle}). Hence:
\begin{eqnarray*}
\omega(z,x)+\omega(t,y) &=& \left(z+t+qJ_{r,k_x} +\frac{qJ_r^\alpha}{N^r},\frac{q}{N^r}\right)\\
&=&
\left(z+t+\xi_\alpha\left(\frac{p_x}{N^{k_x}},\frac{p_y}{N^{k_y}}\right)+\frac{q}{N^r}, \frac{q}{N^r}\right)\\
&=&\omega((z,x)\boxplus(t,y))\text{,}
\end{eqnarray*}
as desired.
\end{case}

\begin{case}
Last, assume $k_x\not=k_y$. Without loss of generality, since our groups are Abelian, we may assume $k_x<k_y$. Now:
\begin{eqnarray*}
\omega(z,x)+\omega(y,t) &=& \left(z+t+\frac{p_xJ^\alpha_{k_x}}{N^{k_x}}+\frac{p_yJ^\alpha_{k_y}}{N^{k_y}},\frac{p_xN^{k_y-k_x}+p_y}{N^{k_y}}\right)\\
&=& \left(z+t+\frac{p_x N^{k_y-k_x} J^\alpha_{k_x}}{N^{k_y}}+\frac{p_yJ^\alpha_{k_y}}{N^{k_y}},\frac{p_xN^{k_y-k_x}+p_y}{N^{k_y}}\right)\\
&=& \left(z+t+\frac{p_xN^{k_y-k_x}(J^\alpha_{k_y}-N^{k_x}J^\alpha_{k_x,k_y})}{N^{k_y}}+\frac{p_yJ^\alpha_{k_y}}{N^{k_y}},\frac{p_xN^{k_y-k_x}+p_y}{N^{k_y}}\right)\\
&=& \left(z+t-p_xJ^\alpha_{k_x,k_y}+\frac{\left(p_xN^{k_y-k_x}+p_y\right)J^\alpha_{k_y}}{N^{k_y}},\frac{p_xN^{k_y-k_x}+p_y}{N^{k_y}}\right)\\
&=&\omega((z,x)\boxplus(t,y))\text{,}
\end{eqnarray*}
as expected.
\end{case}

This completes the proof of that $\omega$ is an isomorphism. Now, $\omega(1,0)=(1,0)$ and $\omega\left(0,\frac{1}{N^k}\right) = \left(\frac{1}{N^k},\frac{1}{N^k}\right)$ for all $k\in\N$. Using Theorem (\ref{Ktheory}), we conclude that tracial states lift to the given map in our theorem.
\end{proof}

Thus, to $\alpha\in\Xi_N$, we can associate a cocycle $\xi_\alpha$ in $H^2(\Q_N,\Z)$ such that $K_0(\algebra_\alpha)$ is given by the extension of $\Q_N$ by $\Z$ associated with $\xi_\alpha$. It is natural to ask how much information the class of $\xi_\alpha$ in $H^2(\Q_N,\Z)$ contains about noncommutative solenoids. This question will be fully answered in the next section, yet we start here by showing that the map $J\in\Z_N\mapsto [\xi_J]\in \mathrm{Ext}(\Q_N,\Z)$ is surjective with kernel $\Z$, where $[\xi]$ is the class of the extensions of $\Q_N$ by $\Z$ (which is Abelian for our cocycles) for the equivalence of extension relation.

First, we recall:
\begin{lemma}\label{ZintoNadicZ}
Let $N\in\N,N>1$. Let $z\in\N$. For all $n\in\N$ we define $\iota(z)_n$ to be the remainder for the Euclidean division of $z$ by $N^n$ in $\Z$. Then $\iota(z)\in\Z_N$ by construction, and there exists $K_z\in\N$ such that $\iota_n(z)=\iota_{K_z}(z)$ for all $n\geq K_z$. Conversely, given any $J\in\Z_N$ which is eventually constant, we can associate the natural number $\zeta(J) = J_K$ where $K$ is the largest natural number such that $J_{K}<J_{K+1}$. One checks easily that $\zeta\circ\iota$ is the identity on $\N$.

The map $\iota$ extends to a group monomorphism from $\Z$ to $\Z_N$. Moreover, if $z<0$ then there exists $K_z\in\N$ such that $\iota_{n+1}-\iota_{n} = N^k(N-1)$ for all $n\geq K_z$. Conversely, if, for some $J\in\Z_N$, there exists $K\in\N$ such that $J_{k+1}-J_{k} = (N-1)N^k$ for all $k\geq K$, then there exists a unique $z\in\Z,z<0$ such that $\iota(z) = J$.
\end{lemma}

\begin{proof}
This is well known.
\end{proof}

We now compute the cohomology relation for our $\Z$-valued cocycles given by $K_0$ groups of noncommutative solenoids:

\begin{theorem}\label{QgroupCohomologous}
Let $N\in\N,N>1$. Let $J=(J_k)_{k\in\N}\in\Z_N$ and $R=(R_k)_{k\in\N}\in\Z_N$. Let $\iota:\Z \rightarrow \Z_N$ be the monomorphism of Lemma (\ref{ZintoNadicZ}). Let $\xi_J,\xi_R$ be the respective $\Z$-valued $2$-cocycle of $\Q_N$ given by Lemma (\ref{QgroupCocycle}). Then $\xi_J$ and $\xi_R$ are cohomologous if and only if $J-R \in \iota(Z)$. This is equivalent to one of the following condition holding:
\begin{itemize}
\item There exists $M\in\N$ such that $J_n-R_n = N^{n-1}(N-1)$ for all $n\geq M$,
\item There exists $M\in\N$ such that $J_n-R_n = N^{n-1}(1-N)$ for all $n\geq M$,
\item There exists $M\in\N$ such that $J_n = R_n$ for all $n\geq M$.
\end{itemize}
In particular, if $N>1$ then there exists nontrivial cocycles of the form $\xi_\alpha$ for some $\alpha\in\Xi_N$.
\end{theorem}

\begin{proof}
Let $\sigma = \xi_J-\xi_R$. For all $n\in\N$, we define $j_n$ and $r_n$ as the unique integers in $\{0,\ldots,N-1\}$ such that $N^n j_n = J_{n+1}-J_n$ and $N^n r_n = R_{n+1}-R_n$. Assume there exists $\psi:\Q_N\rightarrow \Z$ such that for all $x,y\in\Q_N$, we have:
\[
\sigma(x,y) = \psi(x+y)-\psi(x)-\psi(y)\text{.}
\]
Note that $\sigma\left(\frac{p}{N^k},\frac{q}{N^k}\right)=0$ if $p+q$ is not a multiple of $N$, with all fractions written in reduced form in $\Q_N$. Hence, under this condition, we have:
\[
\psi\left(\frac{p}{N^k}+\frac{q}{N^k}\right) = \psi\left(\frac{p}{N^k}\right) + \psi\left(\frac{q}{N^k}\right)\text{.}
\]
We now get:
\[
-j_k+r_k = \psi\left(\frac{1}{N^k}\right) - N\psi\left(\frac{1}{N^{k+1}}\right)\text{,}
\]
so:
\[
\frac{j_k-r_k-\psi\left(\frac{1}{N^k}\right)}{N} = \psi\left(\frac{1}{N^{k+1}}\right)\in\Z
\]
for all $k\in\N$. Hence, for all $k\in\N$ we have:
\[
\psi(1)+(J_k-R_k) \in N^k \Z.
\]
Now:
\begin{eqnarray*}
J_k-R_k &=& \sum_{n=0}^{k-1} N^n (j_n-r_n) 
\end{eqnarray*}
Since $J_1=j_0<N$ and if $J_k<N^k$ then $J_{k+1}=J_k+N^kj_k < N^k +N^{k+1}-N^k = N^{k+1}$, we conclude by induction that $J_k<N^k$ for all $k\in\N$. Hence, $\psi(1)+J_k-R_k\in\Z_k$ implies that either $\psi(1)+J_k-R_k=0$ or $|\psi(1)+J_k-R_k|\geq N^k$.

\begin{case}
Assume first that for all $m\in\N$ there exists $k\in\N$ with $k>m$ such that $\psi(1)+J_k-R_k \geq N^k$. Then, for all $k\in\N$ such that $\psi(1)+J_k-R_k \geq N^k$: 
\begin{equation}\label{cohomologousInequality1}
\psi(1)\geq N^k - J_k + R_k = 1 + \sum_{n=0}^{k-1}N^n (-j_n+r_n+N-1)
\end{equation}
for infinitely many $k\in\N$. Since $-j_n+r_n>-N$, we have $-j_n+r_n+N-1\geq 0$. If $-j_n+r_n+N-1>0$ then the right hand side of Inequality (\ref{cohomologousInequality1}) is unbounded as $k$ is allowed to go to infinity, which is absurd since the left hand side is $\psi(1)$. This implies that there exists $M\in\N$ such that for all $k\geq M$, we have $j_k-r_k = N-1$. Conversely, if there exists $M\in\N$ such that $j_n-r_n=N-1$ for all $n\geq M$, then set $\psi(1) = 1 + \sum_{n=0}^{M-1} N^n(N-1-j_n+r_n)$. We then have:
\begin{eqnarray*}
\psi(1)+J_k-R_k &=& 1+\sum_{n=0}^{M-1} N^n(N-1-j_n+r_n) \\ &+& \sum_{n=0}^{M-1} N^n(j_n-r_n) + \sum_{n=M}^{k-1} N^n(N-1)\\
&=& 1 + \sum_{n=0}^{k-1} \left(N^{k+1}-N^k\right) = N^k
\end{eqnarray*}
as desired.
\end{case}

The cases for $\psi(1)+J_k-R_k\leq N^k$ for infinitely many $k\in\N$, and $\psi(1)=R_k-J_k$ for infinitely many $k\in\N$, are proved similarly.

\end{proof}

\begin{remark}
Let $\alpha\in\Xi_N$ be given such that there exists $\psi\in\Z$ such that $\psi+J_k=N^k$ for all $n\in\N$. Then define the map:
\[
\left| \begin{array}{ccc}
\Z\times \Q_N &\longrightarrow& \mathscr{K}_\alpha\\
\left(z,\frac{p}{N^k}\right)&\longmapsto&\left(z-p\psi+\frac{pJ_k}{N^k},\frac{p}{N^k}\right)
\end{array}\right.
\]
is easily checked to be a group isomorphism. Similar constructions may be used for the other two cases of Theorem (\ref{QgroupCohomologous}).
\end{remark}

The following theorem shows that $K_0$ groups of noncommutative solenoids give all possible Abelian extensions of $\Q_N$ by $\Z$. 

\begin{theorem}\label{QadicExt}
Given any Abelian extension:
\begin{equation}\label{QadicExtExt}
\begin{CD}
0 @>>> \Z @>>> \mathscr{Q} @>>> \Q_N @>>> 0
\end{CD}
\end{equation}
there exists $J\in\Z_N$ such that the extension of $\Z$ by $\Q_N$ given by the cocycle $\xi_J$ of Lemma (\ref{QgroupCocycle}) is equivalent to Extension (\ref{QadicExtExt}). In particular, fixing any $c\in[0,1)$, there exists $\alpha\in\Xi_N$ with $\alpha_0=c$ and such that $\mathscr{Q}$ is isomorphic as a group to $K_0(\algebra_\alpha)$.
\end{theorem}

\begin{proof}
The Pontryagin dual of $\Z_N$ is given by the Pr{\"u}fer $N$-group $\Z(N^\infty)$ defined as the subgroup of $\T$ of all elements of order a power of $N$:
\[
\Z\left(N^\infty\right) = \left\{ \exp\left(2i\pi\frac{p}{N^k}\right) : p\in\Z,k\in\N \right\} 
\]
endowed with the discrete topology. $\Z(N^\infty)$ is also the inductive limit of:
\[
\Z / N\Z \subset \Z / N^2 \Z \subset \Z / N^3 \Z \subset \cdots
\]
and its dual pairing with $\Z_N$ is given by:
\[
\left< J, \frac{p}{N^k} \right> = \exp\left(2i\pi \frac{p(J_k)}{N^{k}} \right)
\]
where $J\in\Z_N$ and $\frac{p}{N^k}\in\Z(N^\infty)$.

\bigskip We note that from the theory of infinite Abelian groups \cite[p. 219]{Fuchs70}, to the short exact sequence of Abelian groups:
\[
\begin{CD}
0 @>>> \Z @>>> \Q_N @>>> \Z(N^\infty) @>>> 0
\end{CD}
\]
there corresponds the Cartan-Eilenberg long exact sequence in $\mathrm{Ext}$ theory for groups:
\[
\xymatrix{
\text{Hom}(\Z(N^\infty),\Z) \ar[r] & \text{Hom}(\Q_N,\Z)  \ar[r] & \text{Hom}(\Z,\Z) \ar[dll] \\
 \mathrm{Ext}(\Z(N^\infty),\Z) \ar[r] & \mathrm{Ext}(\Q_N,\Z)  \ar[r] &\mathrm{Ext}(\Z,\Z).
}
\]
Since $\text{Hom}(\Q_N,\Z)=0$ and $\mathrm{Ext}(\Z,\Z) =0$, we deduce that we have a short exact sequence:
\[
\begin{CD}
0  @>>>\text{Hom}(\Z,\Z) @>>> \mathrm{Ext}(\Z(N^\infty),\Z) @>>>\mathrm{Ext}(\Q_N,\Z)  @>>> 0\text{.}
\end{CD}
\]
Since $\Z(N^\infty)$ is a torsion group, the group  $\mathrm{Ext}(\Z(N^\infty),\Z)$ can be identified with $\text{Hom}(\Z(N^\infty),\Q/\Z)\cong \widehat{\Z(N^\infty)}$ \cite[p. 224]{Fuchs70}, which in turn can be identified with the Pontryagin dual of $\Z(N^\infty)$, namely $\Z_N$. The identification between $\text{Hom}(\Z(N^\infty),\Q/\Z)$ and $\mathrm{Ext}(\Z(N^\infty),\Z)$ is constructed as follows. Let $s$ be a cross-section of $\pi_{\ast}$ with $s(0)=0_{\Q}$ in the short exact sequence:
\[
\begin{CD}
0 @>>> \Z @>>> \Q @>\pi_{\ast}>> \Q/\Z @>>> 0
\end{CD}
\]
where $\pi_{\ast}$ is the natural projection. Any such choice will do, and we take $s(z)=x$ with $x\in\Q\cap [0,1)$ uniquely defined by $x\equiv z \mod \Z$. We can then define the two-cocycle:
\[
\omega \left| \begin{array}{ccl}
\Q / \Z \times \Q / \Z & \longrightarrow & \Z \\
\left(z_1,z_2\right) &\longmapsto & s(z_1) + s(z_2) - s(z_1+z_2)\text{.}
\end{array}\right.
\]
We can now identify $\Z_N$ and  $\mathrm{Ext}(\Z(N^\infty),\Z)$ as follows. For $J=(J_n)_{n\in\N,n>0}$, we define the $\Z$-valued $2$-cocycle of $\Q_N$ by:
\[
\zeta_{J}:\left(\frac{p_1}{N^{k_1}}, \frac{p_2}{N^{k_2}}\right)\in\Q_N^2 \longmapsto \omega\left(J\left[\pi_{\ast}\left(\frac{p_1}{N^{k_1}}\right)\right],J\left[\pi_{\ast}\left(\frac{p_2}{N^{k_2}}\right)\right]\right)\text{.}
\]
We then compute that:
$$s\circ \pi_{\ast}(x)\;=\;[x]\; \text{mod}\;1,\;x\in \Q,$$ where for $x\in \Q,\;[x ] \;\text{mod} \;1$ is defined to be that unique element of $[0,1)$ congruent to $x$ modulo $1$. 

Let us now fix $J\in\Z_N$. As before, we define $(j_n)_{n\in\N}$ by requiring for all $n\in\N,n>0$ that $J_n=\sum_{k=0}^{n-1}N^k j_k$ and  $j_n\in \{0,1,\cdots, N-1\}$. We now calculate that the two-cocycle $\zeta_{J}$ is given as follows:
\begin{eqnarray*}
& &\zeta_{J}\left(\frac{p_1}{N^{k_1}}, \frac{p_2}{N^{k_2}}\right) = \\ & & \mathcal{Z}_J\left(\frac{p_1}{N^{k_1}}, \frac{p_2}{N^{k_2}}\right) - \left\{
\begin{array}{lcl}
\frac{(p_2N^{k_1-k_2}+p_1)}{N^{k_1}}J_{k_1}  \mod 1&\text{if}& k_1>k_2,\\ 
\frac{(p_1N^{k_2-k_1}+p_2)}{N^{k_2}}J_{k_2} \mod 1 & \text{if} & k_1<k_2,\\
\frac{p_1+p_2}{N^{k}} J_{k}\mod 1 & \text{if} & k_1=k_2=k
\end{array}\right.
\end{eqnarray*}
where $\mathcal{Z}_J\left(\frac{p_1}{N^{k_1}}, \frac{p_2}{N^{k_2}}\right)=\left[\frac{p_1J_{k_1}}{N^{k_1}} \mod 1\right] + \left[\frac{p_2J_{k_2}}{N^{k_2}} \mod 1\right]$.
We remark that although each term in the expression defining the cocycle may not be an integer, the combination turns into an integer.
We now claim that if $J\in \Z_N$ is in the image of $\iota:\Z\to \Z_N$ described in the Lemma (\ref{ZintoNadicZ}), then $\zeta_{J}$ is a coboundary. This is to be expected from the short exact sequence giving $\mathrm{Ext}(\Q_N,\Z)$ as a quotient of $\mathrm{Ext}(\Z(N^\infty),\Z).$  In this case, we recall that for $(J_n)_{n\in\N,n>0}=\iota(P)$ for $P\geq 0$, there is $M\in\N$ such that  $J_n=P$ for all $n\geq M$. In that case for all  $k_1, k_2 \geq M$ and $p_1,p_2\in\Z$:
\[
\zeta_{J}(\frac{p_1}{N^{k_1}}, \frac{p_2}{N^{k_2}}))\; =\;
  \left[\frac{p_1P}{N^{k_1}} \mod 1 \right] +\left[\frac{p_2P}{N^{k_2}} \mod 1\right] -\left[\left(\frac{p_1}{N^{k_1}}+\frac{p_2}{N^{k_2}}\right) P \mod 1\right]\text{.}
\]
But this eventually constant sequence is a coboundary, since defining  $\mu_J: \Q_N \to \Z$ by:
\[
\mu_J : \frac{p}{N^k}\mapsto\left[\frac{pP}{N^k}\mod 1\right]-\frac{pP}{N^k} \text{,}
\]
we check that:
\[
\mu_J\left(\frac{p_1}{N^{k_1}}\right)+\mu_J\left(\frac{p_2}{N^{k_2}}\right)-\mu_J\left(\frac{p_1}{N^{k_1}}+\frac{p_2}{N^{k_2}}\right)=\zeta_{J}\left(\frac{p_1}{N^{k_1}}, \frac{p_2}{N^{k_2}}\right)
\]
for all $\frac{p_1}{N^{k_1}},\frac{p_2}{N^{k_2}}\in\Q_N$. Similarly if $J=\iota(P)$ for a negative integer $P$, the statement of Lemma (\ref{ZintoNadicZ}) shows that $j_n=N-1$ for all $n\geq M,$ and one proves in a similar fashion that $\zeta_{J}$ is a coboundary.

We now claim that the  two-cocycle of Lemma (\ref{QgroupCocycle}) (denoted hereafter by $\xi_{J}$) is cohomologous to $\zeta_{J}.$  Recall that $\xi_{J}$ is defined by :
\[
\xi_J \left| \begin{array}{ccl}
\Q_N\times \Q_N & \longmapsto & \Z \\
\left(\frac{p_1}{N^{k_1}},\frac{p_2}{N^{k_2}} \right) & \longmapsto & 
\left\{
\begin{array}{lcl}

-\frac{p_1}{N^{k_1}}\left(J_{k_2}-J_{k_1}\right)&\text{if}&k_2>k_1\\
-\frac{p_2}{N^{k_2}}\left(J_{k_1}-J_{k_2}\right)&\text{if}&k_1>k_2\\
\frac{q}{N^r}\left(J_{k_1}-J_{r}\right)&\text{if}&\wedge\left\{\begin{array}{l}k_1=k_2\\ \frac{p_1}{N^{k_1}}+\frac{p_2}{N^{k_2}}=\frac{q}{N^r}\end{array}\right.\\
\end{array}\right.
\end{array}
\right.
\]
To establish this, we first remark that  for all $\frac{p}{N^k} \in \Q_N$ and for all $m\geq 0$, we have $\left[\frac{p}{N^k} \cdot J_k \mod 1\right] = \left[\frac{pN^m}{N^{k+m}}\cdot J_{k+m}\mod 1\right]$. We establish this by recalling that each $J_k= \sum_{i=0}^{k-1}j_iN^i$ so that $J_{k+m}=\sum_{i=0}^{k+m-1}j_iN^i,$  and the result is an easy computation. 

\bigskip Now consider the following one-cochain, generalizing our definition given earlier on this page:
\[
\mu_J \left| \begin{array}{lcl}
\Q_N & \longrightarrow & \Z \\
\frac{p}{N^k} &\longmapsto& \left[\frac{p}{N^k} J_k \mod 1 \right] - \frac{p}{N^k} J_k
\end{array}\right.
\]
where the $\frac{p}{N^k}$ is taken in reduced form.  Then the cobounding map takes $-\mu_J$ to the following two-coboundary $\delta(-\mu_J)$ on $\Q_N\times \Q_N\to \Z$  given by: 
 \begin{eqnarray*}
\delta(-\mu_J)(\frac{p_1}{N^{k_1}}, \frac{p_2}{N^{k_2}}) &=& \frac{p_1}{N^{k_1}} J_{k_1} 
\\ &-& \left[\frac{p_1}{N^{k_1}} J_{k_1}\mod 1\right] + 
 \frac{p_2}{N^{k_2}} J_{k_2} - \left[\frac{p_2}{N^{k_2}}J_{k_2}\mod 1 \right]\\
  &-& (\frac{p_1}{N^{k_1}}+\frac{p_2}{N^{k_2}}) J_r + \left[ (\frac{p_1}{N^{k_1}}+\frac{p_2}{N^{k_2}}) J_r \mod 1 \right]\text{,}
\end{eqnarray*}
where we want $\frac{p_1}{N^{k_1}}+\frac{p_2}{N^{k_2}} = \frac{q}{N^r}$  in reduced form. Then one verifies that 
\[
\zeta_{J} \delta(-\mu_J) = \xi_{J},
\]
so that the cocycles $\zeta_{J}$ and $\xi_{J}$ are cohomologous. 
\end{proof}

\begin{remark}
Using Theorem (\ref{QadicExt}) and Theorem (\ref{QgroupCohomologous}), we have shown that $\mathrm{Ext}(\Q_N,\Z)$ is isomorphic to $\Z_N/ \Z$ where we identified $\Z$ with $\iota(Z)\subseteq\Z_N$.
\end{remark}

\bigskip We now turn our attention to some properties of the C*-algebras $\algebra_\alpha$ for some special classes of $\alpha$. There are three distinct subclasses of noncommutative solenoids based upon their basic structure:
\begin{definition}
Let $N\in\N,N>1$. Let $\alpha\in\Xi_N$. 
\begin{enumerate}
 \item If $\alpha$ is a periodic sequence (and thus in particular rational), we call $\algebra_\alpha$ a \emph{periodic rational} noncommutative solenoid. These are exactly the nonsimple noncommutative solenoids.
 \item If $\alpha$ is a sequence of rationals, though not periodic, then we call $\algebra_\alpha$ an \emph{aperiodic rational} noncommutative solenoid.
 \item If $\alpha$ is a sequence of irrationals (and thus can never be periodic), then we call $\algebra_\alpha$ an \emph{irrational} noncommutative solenoid.
\end{enumerate}
\end{definition}

We note that simplicity is associated to a form of finiteness, or rationality condition: we need both the (eventual) periodicity of the decimal expansion of the entries of $\alpha$ and the periodicity of $\alpha$ itself. The aperiodic rational case is the more mysterious of the three and an interesting surprise.

\bigskip We start with the case where $\alpha$ is irrational. We use the following well known result \cite{Elliott93} (see also \cite{Dixmier67} for a similar argument used for $\mathrm{AF}$-algebras, which can be applied for $\mathrm{AT}$-algebras as well), whose proof is included for the reader's convenience. We refer to \cite{Lin2001} for the foundation of the theory of \textrm{AT}-algebras. A \emph{circle algebra} is the C*-algebra of $n \times n$ matrix - valued continuous functions on some connected compact subset of $\T$. An \textrm{AT}-algebra is the inductive limit of a sequence of direct sums of circle algebras.

\begin{lemma}\label{AT}
The inductive limit of \textrm{AT}-algebras is \textrm{AT}.
\end{lemma}

\begin{proof}
Let $(A_n)_{n\in\N}$ be a sequence of \textrm{AT}-algebras of inductive limit $A$. To simplify notations, we identify $A_n$ with a subalgebra of $A$ for all $n\in\N$. Let $\varepsilon>0$, $k\in\N$ with $n>0$ and $a_1,\ldots,a_k \in A$. Since $A$ is an inductive limit, there exists $K\in\N$ and $b_1,\ldots,b_k \in A_K$ such that $\|a_j-b_j\|\leq\frac{1}{2}\varepsilon$ for $j=1,\ldots,k$. Now, since $A_K$ is an \textrm{AT}-algebra, there exists $L\in\N$, a finite direct sum $C$ of circle algebras and $c_1,\ldots,c_k \in C$ such that $\|b_j-c_j\|<\frac{1}{2}\varepsilon$ for $j=1,\ldots,k$. Hence, $\|a_j-c_j\|<\varepsilon$. By \cite[Theorem 4.1.5]{Lin2001}, we have characterized $A$ as an \textrm{AT}-algebra.
\end{proof}

\begin{proposition}
Let $N\in\N,N>1$ and $\alpha\in\Xi_N$. If $\alpha_0 \not \in \Q$ (or equivalently, if there exists $k\in\N$ such that $\alpha_k\not\in\Q$), then $\algebra_\alpha$ is a simple \textrm{AT}-algebra of real rank $0$.
\end{proposition}

\begin{proof}
This follows from \cite{Elliott93b}, Theorem (\ref{simplicity}), Lemma (\ref{AT}) and \cite{Lin2001}.
\end{proof}

We now consider $\alpha\in\Xi_N$ ($N\in\N,N>1$) with $\alpha_0$ rational periodic. By Theorem (\ref{simplicity}), $\algebra_\alpha$ is not simple. It is possible to provide a full description of the C*-algebra $\algebra_\alpha$. We denote by $M_q(\C)$ the C*-algebra of $q\times q$ matrices with complex entries, and we denote by $C(X,A)$ the C*-algebra of continuous function from a compact space $X$ to a C*-algebra $A$.

\begin{theorem}\label{rational-nonsimple}
Let $N\in\N$ with $N>1$ and $\alpha\in\Xi_N$. Let $\alpha_0=\frac{p}{q}$ with $p,q\in\N$, nonzero, $p$ and $q$ relatively prime. Assume there exists $k\in\N$ nonzero such that $(N^k-1)\alpha_0\in\Z$, and that $k$ is the smallest such nonzero natural. Let $\lambda = \exp\left(2i\pi \frac{p}{q}\right)$. We define the following two unitaries:
\[
u_\lambda = \left[ \begin{array}{ccccc} 1 \\ & \lambda \\ & & \lambda^2 \\ & & & \ddots&  \lambda^{q-1}  \end{array} \right]
\;\;
v_\lambda = \left[ \begin{array}{ccccc} 0 & \cdots & & 0 & 1 \\
1 & 0 & \cdots & & 0 \\ 0 & \ddots & & & 0 \\ \\ 0 & \cdots & & & 1
\end{array} \right]
\]
and observe that $v_\lambda u_\lambda = \lambda u_\lambda v_\lambda$.
Then $\algebra_\alpha$ is the C*-algebra of continuous sections of a bundle with base space $\solenoid_{N^k}^2$ and fiber $M_q(\C)$. More precisely, $\algebra_\alpha$ is the fixed point of $C(\solenoid^2_{N^k},M_q(C))$ for the action $\rho$ of $\Z / q\Z^2$ given by:
\[
 \rho_{(n,m)}(\zeta) : (z,w) \in \solenoid_N^2\mapsto v_\lambda^{-m}u_\lambda^{-n} \zeta (\lambda^{-n}z,\lambda^{-m}w) u_\lambda^n v_\lambda^m
\]
for $(n,m)\in(\Z/q\Z)^2$ and $\zeta \in C(\solenoid_{N^k}^2,M_q(\C))$.
\end{theorem}

\begin{proof}

By Theorem (\ref{symmetrizer}), our assumption implies that $\alpha$ is $k$-periodic. Let $(\beta_n)_{n\in\N} = (\alpha_0)_{n\in\N}$ --- i.e. $\beta$ is constant, and moreover $\beta\in\Xi_{N^k}$. Let $\theta_k = \varphi_{nk} \circ \ldots \varphi_{(n+1)k-1}$ for all $n\in\N$ where we use the notations of Theorem (\ref{ATorus}). We have:
\[
\algebra_\alpha = \varinjlim(A_{\alpha_{2k}},\varphi_k)=\varinjlim(A_{\beta_{2k}},\theta_k)=\algebra_\beta
\]
as desired. We shall henceforth write $\beta$, by abuse of language, to mean the constant value the sequence $\beta$ takes --- namely $\alpha_0$.

Let $E=C(\solenoid_{N^k}^2,M_q(\C))$ and let $E_\tau$ be the fixed point C*-subalgebra of $E$ for the action $\tau$ of $(\Z/q\Z)^2$. It is well known that the fixed point C*-algebra $E_\tau$ of $\tau$ is *-isomorphic to $A_\beta$.

Let $\varphi:E\rightarrow E$ be defined by setting:
\[
\varphi(\zeta): (z,w)\in\T^2 \mapsto \zeta(z^{(N^k)},w^{(N^k)})
\]
for all $\zeta\in E$. Now, using our assumption that $(N^k-1)\alpha_0\in\Z$ so $\lambda^{(N^k)}=1$, we show that $\varphi$ and $\tau$ commute:
\begin{eqnarray*}\label{rational-commutation}
\tau_{(1,0)}(\varphi(f\otimes A)): (z,w)\in\T^2 &\mapsto& f( (\lambda^{-1}z)^{(N^k)}, w^{(N^k)})\otimes A \\
& = & f (\lambda^{-1} z^{(N^k)} , w^{(N^k)})\otimes A \\
& = & \varphi(\tau_{(1,0)}(f\otimes A))\text{,}
\end{eqnarray*}
for all $f\in C(\T^2)$ and $A\in M_q(\C)$. Hence $\tau_{(1,0)}\circ \varphi = \varphi \circ \tau_{(1,0)}\circ \varphi$ by extending (\ref{rational-commutation}) linearly and by continuity. A similar computation would show that $\tau_{(0,1)}\circ\varphi = \varphi\circ \tau_{(0,1)}$. Hence, $\varphi$ restricts to an endomorphism of $E_\tau$. 
Now, the inductive limit of:
\[
\begin{CD}
E @>\varphi>> E @>\varphi>>  E @>\varphi>> \cdots
\end{CD}
\]
is $C(\solenoid^2_{N^k},M_q(\C))$. Since $\varphi$ and $\tau$ commute, the action $\tau$ extends to the inductive limit by:
\[
\rho_{(p,q)}(\zeta): (z,w)\in\solenoid^2_N \mapsto v_\lambda^{-q} u_\lambda^{-p} \zeta(\lambda^{-p}z, \lambda^{-q}w) u_\lambda^p v_\lambda^q
\]
for all $\zeta\in C(\solenoid_{N^k}^2,M_q(\C))$ and moreover, the inductive limit of:
\[
\begin{CD}
A_\alpha = E_\tau @>\varphi>> A_\alpha @>\varphi>> A_\alpha @>\varphi>> \cdots
\end{CD}
\]
which is $\algebra_\alpha$ by Theorem (\ref{ATorus}) is also the fixed point of $C(\solenoid_{N^k}^2, M_q(\C))$ by the action $\rho$ of $(\Z/q\Z)^2$ on $C(\solenoid_{N^k}^2,M_q(\C))$. Hence our theorem.
\end{proof}

We note that the proof of Theorem (\ref{rational-nonsimple}) shows that the embeddings from Theorem (\ref{ATorus}) map from and to the centers of the rotation C*-algebras. This is in contrast with the situation when $\alpha_0$ is rational but $\alpha$ is not pseudo-periodic, which illustrates why the associated noncommutative solenoids are simple.

\section{The isomorphism problem}

Our classification of noncommutative solenoids is based on our computation of their $K$-theory. We start with the following simple observation:

\begin{lemma}\label{NadicIso}
Let $\sigma:\Q_N\rightarrow \Q_N$ be a group isomorphism. Then there exists $p\in\Z$ with $p\mid N$ and $p\not\in\{-N,N\}$ and $k\in\N$ such that $\sigma(1) = \frac{p}{N^k}$. Consequently $\sigma\left(\frac{1}{N^r}\right) = \frac{p}{N^{k+r}}$ for all $r\in\N$.
\end{lemma}

\begin{proof}
Let us write $\sigma(1) = \frac{pq}{N^k}$ in its reduced form, with $q$ relatively prime with $N$ and nonnegative. Note that as $\sigma$ is an isomorphism, $pq\not=0$ and moreover, there exists $x\in\Q_N$ such that $\sigma(x)=\frac{p}{N^k}$ and we must have $qx = 1$. This contradicts the relative primality of $N$ and $q$. 
\end{proof}

We now obtain the main result of our paper. We fully characterize the isomorphism classes of noncommutative solenoids based on the multipliers of adic rationals.

\begin{theorem}\label{classification}
Let $N,M\in\N$ with $N>1$ and $M>1$. Let $\alpha\in\Xi_N$ and $\beta\in\Xi_M$. The following assertions are equivalent:
\begin{enumerate}
\item The C*-algebras $\algebra_\alpha$ and $\algebra_\beta$ are *-isomorphic,
\item The integers $N$ and $M$ have the same set of prime factor. Let $R$ be the the greatest common divisor of $N$ and $M$, and let $\mu = \frac{N}{R}$ and $\nu=\frac{M}{R}$. Set $\alpha_n'=\mu^n \alpha_n \mod 1$ and $\beta_n'=\nu^n \beta_n \mod 1$ for all $n\in\N$ and note $\alpha',\beta'\in\Xi_R$. There exists $\Lambda\in\primeseq$ and $\gamma\in \Xi_\Lambda$ such that both $\alpha'$ and $\gamma$ have a common subsequence, and $\beta'$ or $-\beta'=(1-\beta_n')_{n\in\N}$  has a common subsequence with $\gamma$. Moreover, $\{\Lambda_n : n\in \N \}$ is the set of prime factors of $R$.
\end{enumerate}
\end{theorem}

\begin{proof}
Assume that there exists $\Lambda\in\primeseq$ and $\gamma\in\Xi_\Lambda$ such that $\alpha$ and $\beta$ have subsequences which are also subsequences of $\gamma$. Then a standard intertwining argument shows that $\algebra_\alpha$ and $\algebra_\beta$ are *-isomorphic to $\algebra_\gamma$. Moreover, for any irrational rotation algebra $A_\theta$, we have that $A_\theta$ is *-isomorphic to $A_{-\theta}$. Hence, $\algebra_\beta$ and $\algebra_{-\beta}$ are *-isomorphic as well. 

Now, let $N=\mu R$ and assume the set of prime factors in $\mu$ is a subset of the set of prime factors of $N$. Set $\alpha'_n = \mu^n \alpha_n$ for all $n\in\N$. 

First, it is straightforward to show that if $N$ and $R$ have the same set of prime factors, then $\Q_N$ and $Q_R$ are isomorphic.

Second, for all $n\in\N$ we have:
\begin{eqnarray*}
R \alpha'_{n+1} &\equiv& R \mu^{n+1} \alpha_{n+1} \mod \Z\\
&\equiv&  \mu^n N\alpha_{n+1} \mod \Z \\
&\equiv& \mu^n\alpha_{n} \mod \Z \\
&\equiv& \alpha'_n \mod \Z\text{.}
\end{eqnarray*}
Hence $\alpha'\in\Xi_R$.

Third, given $\frac{p_j}{R^{k_j}}=\frac{p_j\mu^{k_j}}{N^{k_j}}\in\Q_R$ for $j=1,2,3,4$, we have:
\begin{eqnarray*}
& &\Psi_{\alpha'}\left(\left(\frac{p_1}{R^{k_1}},\frac{p_2}{R^{k_2}}\right),\left(\frac{p_3}{R^{k_3}},\frac{p_4}{R^{k_4}} \right) \right) =\\
&= & \exp\left(2i\pi \left(\alpha'_{k_1+k_4}p_1p_4 \right)\right)\\
&= & \exp\left(2i\pi \left(\alpha_{k_1+k_4}(\mu^{k_1}p_1 \mu^{k_4}p_4 \right)\right)\\
&= & \Psi_\alpha\left(\left(\frac{p_1\mu^{k_1}}{N^{k_1}},\frac{p_2\mu^{k_2}}{N^{k_2}}\right),\left(\frac{p_3\mu^{k_3}}{N^{k_3}},\frac{p_4\mu^{k_4}}{N^{k_4}} \right) \right)\text{.}
\end{eqnarray*}
Hence, $\Psi_\alpha=\Psi_\alpha'$. Consequently, $\algebra_{\alpha'}=\algebra_{\alpha}$. This concludes the proof that (2) implies (1).

\bigskip Conversely, let $\theta : \algebra_\alpha \rightarrow \algebra_\beta$ be a *-isomorphism.  We shall use the notations introduced in Theorem (\ref{Ktheory}). If $\tau$ is a tracial state of $\algebra_\beta$ then $\tau\circ \theta$ is a tracial state on $\algebra_\alpha$. Denote, respectively, by $\tau_\alpha$ and $\tau_\beta$ the lift of a tracial state of $\algebra_\alpha$ and $\algebra_\beta$, and note that by Theorem (\ref{traces}), the choices of tracial state is irrelevant.

By functoriality of $K$-theory, we obtain the following commutative diagram:
\begin{equation}\label{commutative-K-diagram}
\xymatrix{
{K_0(\algebra_\alpha)} \ar[rr]^{K_0(\theta)}  \ar[dr]_{\tau_\alpha} & & {K_0(\algebra_\beta)} \ar[dl]^{\tau_\beta} \\
   & {\R} & 
}
\end{equation}
where $K_0(\theta)$ is the group isomorphism induced by $\theta$. To ease notations, let us write $\sigma = K_0(\theta)$.

\bigskip Our first observation is that $\tau_\beta\circ\sigma(1,0) = \tau_\alpha(1,0) = 1$, which implies that $\sigma(1,0)=(1,0)$. 

Let $\pi_\beta : \mathscr{K}_\beta \rightarrow \Q_M$ be defined by $\pi_\beta\left(z+\frac{pJ_k^\beta}{M^k},\frac{p}{M^k}\right) = \frac{p}{M^k}$. It is easily checked that $\pi_\beta$ is a group epimorphism. Moreover, $\ker\pi_\beta = \{(z,0):z\in\Z\}$. Consequently, if $z,z'\in\Z$, since $\sigma\left(z+\frac{pJ_k^\alpha}{N^k},\frac{p}{N^k}\right) = \sigma(z,0)+\sigma\left(\frac{pJ_k^\alpha}{N^k},\frac{p}{N^k}\right)$, we observe that:
\[
\pi_\beta\left(\sigma\left(z+\frac{pJ_k^\alpha}{N^k},\frac{p}{N^k}\right)-\sigma\left(z'+\frac{pJ_k^\alpha}{N^k},\frac{p}{N^k}\right)\right) = 0\text{.}
\]
Consequently, we have the following commuting diagram:
\[
\xymatrix{
{\mathscr{K}_\alpha} \ar[r]^{\sigma}\ar[d]_{\pi_\alpha} & \mathscr{K}_\beta\ar[d]^{\pi_\beta} \\
{\Q_N} \ar[r]^{f} & {\Q_M}\text{.}
}
\]
with $f:\Q_N\rightarrow \Q_M$ defined by setting $f\left(\frac{p}{N^k}\right) = \pi_\beta\circ\sigma\left(\frac{pJ_k^\alpha}{N^k},\frac{p}{N^k}\right)$. In particular, $f$ is a group isomorphism, so the set of prime factors of $N$ and $M$ are the same and $\Q_N=\Q_M$. As we showed in the first half of this proof, and using the definition of our Theorem, $\alpha',\beta'\in\Xi_R$ and $\Q_N=\Q_R$ where $R$ is the greatest common divisor of $N,M$ and $\algebra_{\alpha'} = \algebra_{\alpha}$ while $\algebra_{\beta'}=\algebra_{\beta}$. We shall henceforth work within $\Q_R$ with $\alpha'$ and $\beta'$.

Let $p\in\Z,k\in\N$ be defined so that $f(1)=\frac{p}{R^k}$ and $\frac{p}{R^k}$ is in reduced form, with $p \mid R$ and $p\not\in\{-R,R\}$ by Lemma (\ref{NadicIso}) . Since $f$ is an isomorphism, we have $f\left(\frac{1}{R^n}\right) = \frac{p}{R^{k+n}}$ for all $n\in\N$. Using the notation $\Omega(R)$ for the number of prime factors of $R$, let $\Lambda\in\primeseq$ be defined as a periodic sequence of period $\Omega(R)$ such that $\Lambda_{\Omega(R)-1-j}=\Lambda(p)_{\Omega(R)-1-j}$ for $j=0,\ldots,\Omega(p)-1$ and $\pi_{\Omega(R)}(\Lambda)=R$. Any of the $(\Omega(R)-\Omega(p))!$ possible choices of order for the first $\Omega(R)-\Omega(p)$ values of $\Lambda$ can be used, and we assume we pick one in the rest of this proof. We can visualize $\Lambda$ as:
\[
\Lambda = \left( \underbrace{\Lambda_0, \Lambda_1,\cdots,\overbrace{\Lambda_{\Omega(R)-\Omega(p)},\cdots,\Lambda_{\Omega(R)-1}}^{\text{product $=p$}} }_{\text{product $=R$}}, \underbrace{\Lambda_{\Omega(R)},\cdots,\Lambda_{2\Omega(R)-1}}_{\text{equal to previous $\Omega(R)$ terms}}, \cdots \right)
\]

Let $\gamma$ be the (unique) extensions of $\beta'$ to $\Xi_\Lambda$. Thus $\gamma_{\Omega(R)n} = \beta'_n$ for all $n\in\N$. Now, for any $n\in\N$, there exists $p_n\in\Z$ such that \[ \sigma\left(\frac{J_n^{\alpha'}}{R^n},\frac{1}{R^n}\right) = \left(p_n+\frac{p J_{n+k}^{\beta'}}{R^{n+k}},\frac{p}{R^{n+k}}\right)\text{.} \]
Using the computation of the traces on $K_0$ in Theorem (\ref{Ktheory}) and the commutative diagram (\ref{commutative-K-diagram}), and noting that if $r=\Omega(p)$ then $p\beta_n' = p\gamma_{n\Omega(R)} = \gamma_{n\Omega(R)-r}$ by definition of $p$, $\Lambda$ and $\gamma$, we thus have:
\[
\alpha_n' = p \beta_{n+k}' + p_n \equiv \mathrm{sign}(p) \gamma_{(n+k)\Omega(N)-r} \mod \Z \text{.}
\]

Thus, $\mathrm{sign}(p)\alpha'$ is a subsequence of $\gamma$, and $\beta$ is a subsequence of $\gamma$ by construction. This concludes the proof of (1) implies (2).
\end{proof}

\begin{corollary}
Let $N,M$ be prime numbers. Let $\alpha\in\Xi_N$ and $\beta\in\Xi_M$. Then the following assertions are equivalent:
\begin{enumerate}
 \item The noncommutative solenoids $\algebra_\alpha$ and $\algebra_\beta$ are *-isomorphic,
 \item We have $N=M$ and one of the sequence $\alpha$ or $\beta$ is a truncated subsequence of the other.
\end{enumerate}
\end{corollary}

\begin{proof}
If $N=M$ and $\alpha$ is a truncated subsequence of $\beta$ then $\algebra_\alpha$ and $\alpha_\beta$ are trivially *-isomorphic. The same holds if $\beta$ is a truncated subsequence of $\alpha$.

Conversely, assume $\algebra_\alpha$ and $\algebra_\beta$ are *-isomorphic. Then as $N$ and $M$ are prime, so by Theorem (\ref{classification}) we have $N=M$. Moreover, there exists a sequence $\gamma \in \Xi_N$ such that both $\alpha$ and $\beta$ are subsequences of $\gamma$. Now, since $\alpha,\beta,\gamma\in\Xi_N$, this implies that for some $n,n'\in\N$ we have $\alpha_{j} = \gamma_{n+j}$ and $\beta_{j}=\gamma_{n'+j}$ for all $j\in\N$. This shows that either $\alpha$ is a truncated subsequence of $\beta$ (if $n'\leq n$) or $\beta$ is a truncated subsequence of $\alpha$.
\end{proof}

Theorem (\ref{classification}) relies on the invariant $\algebra_\alpha \; (\alpha\in\Xi_N) \mapsto \left(K_0(\algebra_\alpha),\tau_\alpha\right)$ where $\tau_\alpha$ is the unique map given by lifting any tracial state of $\algebra_\alpha$ to its $K_0$ group. We would like to add an observation regarding the information on noncommutative solenoids one can read from the $K_0$ group seen as an Abelian extension of $\Z$ by $\Q_N$ rather than as a group alone. We fix $N\in\N,N>1$.

First, note that given $\alpha\in\Xi_N$, the pair $(K_0(\algebra_\alpha),[1])$, where $[1]$ is the $K_0$-class of the identity of $\algebra_\alpha$, we can construct an Abelian extension of $\Z$ by $\Q_N$ by defining $\iota : z \in \Z \mapsto z[1]$ and noting that $K_0(\algebra_\alpha)/i(\Z)$ is isomorphic to $\Q_N$.

Now, consider $\alpha,\beta\in\Xi_N$ such that there exists a (unital) *-isomorphism $\psi:\algebra_\alpha\rightarrow \algebra_\beta$. Then the following diagram commutes:
\begin{equation}\label{WeakEquivalence}
\xymatrix{
0 \ar[r] & {\Z} \ar[r]^\iota \ar@{=}[d] & {\mathscr{Q}_\alpha=K_0(\algebra_\alpha)} \ar[r] \ar^{K_0(\psi)}[d] & {\Q_N} \ar[r] \ar[d]^\sigma & 0\\
0 \ar[r] & {\Z} \ar[r]^\iota & {\mathscr{Q}_\beta=K_0(\algebra_\beta)} \ar[r] & {\Q_N} \ar[r] & 0
}
\end{equation}
since $\psi$ is unital, and where the arrow $\sigma$ is defined and proven to be an isomorphism by standard diagram chasing arguments.

Conversely, we say that two Abelian extensions of $\Z$ by $\Q_N$ such that there exists a commutative diagram of the form Diagram (\ref{WeakEquivalence})  are \emph{weakly equivalent} (note that Theorem (\ref{QadicExt}) shows that any such extension can be obtained using the $K$-theory of noncommutative solenoids). Note that weakly equivalent extensions are isomorphic but not necessarily equivalent as extensions. The difference is that we allow for an automorphism $\sigma$ of $\Q_N$. This reflects, informally, that according to Theorem (\ref{classification}), the noncommutative solenoid $\algebra_\alpha$ only partially determines $\alpha$.

Now, given two equivalent Abelian extensions of $\Z$ by $\Q_N$, if one is weakly equivalent to some other extension, then so is the other. Hence, weakly equivalence defines an equivalence relation $\equiv$ on $\mathrm{Ext}(\Q_N,\Z)$ such that if $\alpha,\beta\in\Xi_N$ give rise to *-isomorphic noncommutative solenoids, then the associated Abelian extensions of cocycle $\xi_\alpha$ and $\xi_\alpha$ (see Lemma (\ref{QgroupCocycle})) in $\mathrm{Ext}(\Q_N,\Z)$ are equivalent for $\equiv$. According to Theorem (\ref{QadicExt}), the group $\mathrm{Ext}(\Q_N,\Z)$ is isomorphic to the quotient $\Z_N/\Z$ of the group $\Z_N$ of $N$-adic integers by the group $\Z$ of integers. Using Theorem (\ref{classification}), and for $N$ prime, we easily see that the relation induced by $\equiv$ on $\Z_N/\Z$ is given by:
\[
[J_n]_{n\in\N} \equiv [R_n]_{n\in\N} \iff \exists k \in \N \;\; \left(N^k J - R \in \Z \right) \vee \left(N^k R - J \in \Z\right)
\]
where $[J]$ is the class in $\Z_N / \Z$ of $J\in\Z_N$ and $N^k J$ is the sequence $(N^kJ_n)_{n\in\N}$ for any $J\in\Z_N$.

Hence, in conclusion, for a given $\alpha\in\Xi_N$, the data $\left(K_0(\algebra_\alpha),[1],\alpha_0\right)$ where $\alpha_0$ is the trace of any Rieffel-Powers projection in $\algebra_\alpha$, is a complete invariant for $\algebra_\alpha$. Indeed, $(K_0(\algebra_\alpha),[1])$ determines a cocycle $\xi_J$ in $H^2(\Q_N,\Z)$, up to the equivalence $\equiv$, and we can recover $\alpha$ up to a shift using the value $\alpha_0$.

\bigskip Our classification result has the following interesting dynamical application. The following is easily seen:

\begin{corollary}
Let $N,M\in\N$ with $N,M>1$. Let $\alpha\in\Xi_N$ and $\beta\in\Xi_M$. If any of the following assertion holds:
\begin{enumerate}
\item $N$ and $M$ have distinct set of prime factors,
\item For any $\Lambda\in\primeseq$ such that $\{ \Lambda_n : n\in\N \}$ is the set of prime factors of $N$, there is no $\gamma\in\Xi_\Lambda$ such that $\alpha'$ and $\beta'$ are subsequences of $\gamma$ and no $\gamma\in\Xi_\Lambda$ such that $\alpha'$ and $-\beta'$ are subsequences of $\gamma$, where $\alpha',\beta'\in\Xi_R$ are defined as in Theorem (\ref{classification}),
\end{enumerate}
then the actions $\theta^\alpha$ and $\theta^\beta$ of, respectively, $\Q_N$ on $\solenoid_N$ and $\Q_M$ on $\solenoid_M$ are not topologically conjugate.
\end{corollary}

\bibliographystyle{amsplain}
\bibliography{../thesis.bib}

\providecommand{\bysame}{\leavevmode\hbox to3em{\hrulefill}\thinspace}
\providecommand{\MR}{\relax\ifhmode\unskip\space\fi MR }
\providecommand{\MRhref}[2]{%
  \href{http://www.ams.org/mathscinet-getitem?mr=#1}{#2}
}
\providecommand{\href}[2]{#2}
\begin{thebibliography}{10}

\bibitem{Baggett10}
{L}. {B}aggett, {N}. {L}arsen, {J}. {P}acker, {I}. {R}aeburn, and {A}.
  {R}amsay, \emph{Direct limits, multiresolution analyses, and wavelets},
  Journal of Funct. Anal. \textbf{2010} (258), 2714--2738.

\bibitem{Baggett11}
{L}. {B}aggett, {K}. {M}errill, {J}. {P}acker, and {A}. {R}amsay,
  \emph{Probability measures on solenoids corresponding to fractal wavelets},
  Trans. Amer. Math. Soc. (To appear).

\bibitem{Connes80}
{A}. {C}onnes, \emph{{C*}--alg{\`e}bres et g{\'e}om{\'e}trie differentielle},
  {C}. {R}. de l'academie des Sciences de Paris (1980), no.~series A-B, 290.

\bibitem{Dixmier67}
{J}. {D}ixmier, \emph{Some {C*-}algebras considered by {G}limm}, Journal of
  Funct. Anal. \textbf{1} (1967), 182--203.

\bibitem{Dutkay06}
{D}. {D}utkay, \emph{Low-pass filters and representations of the baumslag
  solitar group}, Trans. Amer. Math. Soc. \textbf{358} (2006), 5271--5291.

\bibitem{DutkayJorgensen06}
{D}. {D}utkay and {P}. {J}orgensen, \emph{Wavelets on fractals}, Rev. Mat.
  Iberoam. (2006), 131--180.

\bibitem{DutkayJorgensen07}
\bysame, \emph{Analysis of orthogonality and of orbits in affine iterated
  function systems}, Math. Z. \textbf{256} (2007), 801--823.

\bibitem{Echterhoff95}
{S}. {E}chterhoff and {J}. {R}osenberg, \emph{Fine structure of the {M}ackey
  machine for actions of {A}belian groups with constant {M}ackey obstruction},
  Pacific J. Math. \textbf{170} (1995), 17--52.

\bibitem{EffrosHahn67}
{E}. {E}ffros and {F}. {H}ahn, \emph{Locally compact transformation groups and
  {$C^{\ast}$}- algebras}, vol.~75, American Mathematical Society, 1967.

\bibitem{Elliott93b}
{G.} {E}lliott and {D}. {E}vans, \emph{Structure of the irrational rotation
  {$C^\ast$}-algebras}, Annals of Mathematics \textbf{138} (1993), 477--501.

\bibitem{Elliott93}
{G}.~{A}. {E}lliott, \emph{On the classification of {$C\sp{\ast} $}-algebras of
  real rank zero}, J. Reine Angew. Math. \textbf{443} (1993), 179--219.

\bibitem{Fuchs70}
{L}. {F}uchs, \emph{Infinite abelian groups, {V}olume {I}}, Academic Press,
  1970.

\bibitem{Hoegh-Krohn81}
R.~Hoegh-Krohn, M.~B. Landstad, and E.~Stormer, \emph{Compact ergodic groups of
  automorphisms}, {A}nnals of {M}athematics \textbf{114} (1981), 75--86.

\bibitem{Kaplansky69}
{I}. {K}aplansky, \emph{Infinite abelian groups}, University of Michigan Press,
  1969.

\bibitem{Kleppner65}
A.~Kleppner, \emph{Multipliers on {A}belian groups}, Mathematishen Annalen
  \textbf{158} (1965), 11--34.

\bibitem{Lin2001}
{H}. {L}in, \emph{An introduction to the classification of amenable
  {$C^\ast$}-algebras}, World Scientific, 2001.

\bibitem{Packer92}
{J}. {P}acker and {I}. {R}aeburn, \emph{On the structure of twisted group
  {$C^\ast$-}algebras}, Trans. Amer. Math. Soc. \textbf{334} (1992), no.~2,
  685--718.

\bibitem{Pimsner80}
{M}. {P}imsner and {D}.~{V}. Voiculescu, \emph{Exact sequences for {K}-groups
  and {Ext}-groups of certain crossed product {C*}-algebras}, Journal of
  Operator Theory \textbf{4} (1980), no.~1, 93--118.

\bibitem{Rieffel81}
{M}.~{A}. {R}ieffel, \emph{{C*}-algebras associated with irrational rotations},
  Pacific Journal of Mathematics \textbf{93} (1981), 415--429.

\bibitem{Rieffel88}
M.~A. {R}ieffel, \emph{Projective modules over higher-dimensional
  non-commutative tori}, {C}an. {J}. {M}ath. \textbf{XL} (1988), no.~2,
  257--338.

\bibitem{Yin86}
{H}.-{S}. {Y}in, \emph{A simple proof of the classification of rational
  rotation {C*-}algebras}, Proc. Amer. Math. Soc. \textbf{98} (1986), 469--470.

\end{thebibliography}
\vfill
\end{document}